\documentclass[12pt]{article}

\usepackage{amssymb}%
\usepackage{amsmath}%
\usepackage{amsthm}%

\usepackage{xcolor}%

\allowdisplaybreaks%

{\theoremstyle{plain}%
 \newtheorem{theorem}{Theorem}

}
{\theoremstyle{remark}

}
{\theoremstyle{definition}

\newtheorem{example}{Example}
}

\begin{document}

\begin{center}
{\large Hypergeometric accelerations with shifted indices}

  \  

{\textsc{John M. Campbell} } 

  \  

{\footnotesize Department of Mathematics and Statistics}

{\footnotesize Dalhousie University}

{\footnotesize Halifax,    Nova   Scotia,  Canada}

  \

{\footnotesize {\tt jmaxwellcampbell@gmail.com}}

  \   

{\footnotesize ORCID: 0000-0001-5550-2938}

 \ 

\end{center}

\begin{abstract}
 Chu and Zhang, in 2014, introduced hypergeometric transforms derived through the application of an Abel-type summation lemma to Dougall's 
 ${}_{5}H_{5}$-series. These transforms were applied by Chu and Zhang to obtain accelerated rates of convergence, yielding rational series related to the 
 work of Ramanujan and Guillera. We apply a variant of an acceleration method due to Wilf using what we refer to as \emph{shifted indices} for 
 Pochhammer symbols involved in our first-order, inhomogeneous recurrences derived via Zeilberger's algorithm, to build upon Chu and Zhang's 
 accelerations, recovering many of their accelerated series and introducing many inequivalent series for universal constants, including series of Ramanujan 
 type involving linear polynomials as summand factors, as in Ramanujan's series for $\frac{1}{\pi}$. 
\end{abstract}



\section{Introduction}
 An important contribution in the development of techniques for efficiently computing universal constants such as $\pi$ is due to Chu and Zhang in 2014 
 \cite{ChuZhang2014} and is related to many subsequent applications concerning special functions, number theory, and difference equations. As in the 
 work of Chu and Zhang \cite{ChuZhang2014}, we highlight especially groundbreaking discoveries due to Ramanujan \cite{Ramanujan1914}, the Borwein 
 brothers \cite{BorweinBorwein1987}, and the Chudnovsky brothers \cite{ChudnovskyChudnovsky1988} in the history of $\pi$ formulas, with a number 
 of Ramanujan's famous formulas for $\frac{1}{\pi}$ having been proved as special cases of acceleration identities introduced by Chu and Zhang 
 \cite{ChuZhang2014}. The series accelerations introduced by Chu and Zhang were derived via the modified Abel lemma on summation by parts 
 applied with Dougall's ${}_{5}H_{5}$-series, and many further summation techniques have subsequently been introduced by Chu 
 \cite{Chu2023,Chu2018,Chu2020Alternating,Chu2021Further,Chu2020Infinite,Chu2021Infinite,Chu2021Ramanujan,Chu2019} to produce fast 
 converving series for fundamental constants. The computer proof-based machinery associated with \emph{Wilf--Zeilberger (WZ) theory} 
 \cite{PetkovsekWilfZeilberger1996} is not involved in these past references but provides a powerful tool in the experimental discovery of series related 
 to Ramanujan's work, as in the groundbreaking discoveries due to Guillera 
 \cite{Guillera2018Dougall,Guillera2006,Guillera2008,Guillera2013,Guillera2011,Guillera2010,Guillera2018Proofs,Guillera2002}. This leads us to consider how 
 discrete difference equations-based series accelerations derived using tools associated with WZ theory, as in our recent work 
 \cite{CampbellLevrieunpublished}, could be applied to build upon the results of Chu and Zhang \cite{ChuZhang2014}. In this paper, we apply what may 
 be regarded as a variant of an acceleration method due to Wilf \cite{Wilf1999} to provide multiparameter families of hypergeometric transforms. 
 Through systematic computer searches through combinations of rational parameter values, this has led us to recover many of the formulas proved by Chu 
 and Zhang \cite{ChuZhang2014} and to experimentally discover many new and accelerated series for fundamental constants. Our new series are 
 inequivalent, even up to integer differences of Pochhammer symbols, to previously introduced series as in the work of Chu et al.\ 
 \cite{Chu2023,Chu2018,Chu2020Alternating,Chu2021Further,Chu2020Infinite,Chu2021Infinite,Chu2021Ramanujan,Chu2019,ChuZhang2014}. 
 Moreover, our technique based on Zeilberger's algorithm \cite[\S6]{PetkovsekWilfZeilberger1996} may be seen as providing a more versatile way of 
 obtaining hypergeometric accelerations, compared to the reliance in \cite{ChuZhang2014} on Dougall's ${}_{5}H_{5}$-series. 

 The {$\Gamma$-function} \cite[\S2]{Rainville1960} is of basic importance in the use of special functions and may be defined for $\Re(x) > 0$ with the 
 Euler integral so that $\Gamma(x)=\int_{0}^{\infty} u^{x-1}e^{-u}\,du$. This leads us to define the \emph{Pochhammer symbol} so that $(a)_{k} = 
 \frac{\Gamma (a+k)}{\Gamma (a)}$. For $k \in \mathbb{N}_{0}$, this provides the rising factorial function, with $(a)_{k} = a(a + 1) \cdots (a + k - 
 1)$. A main object of study in this paper is given by what we have referred to as a \emph{Chu-style series} 
 \cite{Campbell2023,CampbellLevrieunpublished}, and this leads us to make use of the notational convention such that 
\begin{equation*}
\left[ \begin{matrix} \alpha, \beta, \ldots, \gamma \vspace{1mm} \\ 
 A, B, \ldots, C \end{matrix} \right]_{n} = \frac{ (\alpha)_{n} (\beta)_{n} 
 \cdots (\gamma)_{n} }{ (A)_{n} (B)_{n} \cdots (C)_{n}}. 
\end{equation*}
 Following the notation in Chu and Zhang's work \cite{ChuZhang2014}, we highlight the below listed series for $\frac{1}{\pi}$ famously introduced by 
 Ramanujan \cite{Ramanujan1914}. Each of the following formulas was proved by Chu and Zhang \cite{ChuZhang2014} via their acceleration based on the 
 Dougall summation identity: 
\begin{align}
 \frac{4}{\pi} & = 
 \sum_{k = 0}^{\infty} \left(\frac{1}{4} \right)^{k} 
 \left[ \begin{matrix} 
 \frac{1}{2}, \frac{1}{2}, \frac{1}{2} \vspace{1mm} \\ 
 1, 1, 1 
 \end{matrix} \right]_{k} \left( 6 k + 1 \right), \label{originalRamanujan1} \\ 
 \frac{8}{\pi} & = 
 \sum_{k = 0}^{\infty} \left(-\frac{1}{4} \right)^{k} \left[ \begin{matrix} 
 \frac{1}{4}, \frac{1}{2}, \frac{3}{4} \vspace{1mm} \\ 
 1, 1, 1 
 \end{matrix} \right]_{k} \left( 20 k + 3 \right), \\ 
 \frac{16}{\pi} & = 
 \sum_{k = 0}^{\infty} \left(\frac{1}{64} \right)^{k} \left[ \begin{matrix} 
 \frac{1}{2}, \frac{1}{2}, \frac{1}{2} \vspace{1mm} \\ 
 1, 1, 1 
 \end{matrix} \right]_{k} \left( 42 k + 5 \right). \label{originalRamanujan3} 
\end{align}

 The accelerations introduced by Chu and Zhang \cite{ChuZhang2014} rely on recursions for $$ \Omega(a;b,c,d,e) := 
 \sum_{k=0}^{\infty} (a + 2k) \left[ \begin{matrix} 
 b, c, d, e \vspace{1mm} \\ 
 1+a-b, 1-a-c, 1 - a - d, 1 + a - e 
 \end{matrix} \right]_{k} $$ 
 derived via an Abel-type summation lemma, 
 letting $\Re(1 + 2 a - b - c - d - e) > 0$. 
 By expressing $\Omega(a; b, c, d, e)$ 
 in terms of $\Omega(a + n_{a}; b + n_b, c + n_c, d + n_d, e + n_e)$ 
 for integer parameters $n_a$, $n_b$, $n_c$, $n_d$, and $n_e$, 
 if it is possible to obtain a two-term recursion of the form 
\begin{equation}\label{twotermOmega}
 \Omega(a;b,c,d,e) = R_{1} + R_2 \, \Omega(a + n_a; b + n_b, c + n_c, d + n_d, e + n_e) 
\end{equation}
 for rational functions $R_1$ and $R_2$ with the parameters $a$, $b$, $c$, $d$, and $e$ as the arguments of $R_{1}$ and $R_2$, then the repeated 
 application of the recursive relation in \eqref{twotermOmega} produces accelerations for special cases covered by Chu and Zhang \cite{ChuZhang2014}. 
 Instead of making use of Abel-type series rearrangements, we apply two-term hypergeometric recursions through a variant of an acceleration 
 method that is due to Wilf \cite{Wilf1999} and that was was applied in our recent work \cite{CampbellLevrieunpublished}; see also the related work by 
 Mohammed \cite{Mohammed2005} and by Hessami Pilehrood and Hessami Pilehrood \cite{HessamiPilehroodHessamiPilehrood2008}. Unlike these past 
 references on the use of two-term hypergeometric recursions to produce series accelerations, our method relies on what we refer to as \emph{shifted 
 indices} for Pochhammer symbols, as described in Section \ref{sectionmain} below. To begin with, we highlight, as motivating results in 
 Section \ref{sectionMotivating}, a number of new formulas produced via our shifted acceleration method. 

\section{Motivating results}\label{sectionMotivating}
 A number of results introduced and proved in this paper are given by formulas that may be described as being of \emph{Ramanujan type}, noting the 
 resemblance between Ramanujan's formulas in \eqref{originalRamanujan1}--\eqref{originalRamanujan3} and the following new results: 
\begin{align}
 \frac{5 \pi }{4 \sqrt{3}}  & = \sum_{k = 0}^{\infty} \left(\frac{1}{4} \right)^{k} \left[ \begin{matrix} \frac{1}{3}, 1, \frac{5}{3} 
 \vspace{1mm} \\ \frac{7}{6}, \frac{3}{2},  \frac{11}{6} \end{matrix} \right]_{k} \left( 3 k+2 \right), \label{RamanujanType1} \\ 
 \frac{15 \sqrt{2}}{16}  & = \sum_{k=0}^{\infty} \left(\frac{1}{4} \right)^{k} \left[ \begin{matrix} \frac{3}{4}, \frac{3}{2}, 
 \frac{3}{2} \vspace{1mm} \\ 1, \frac{13}{8},  \frac{17}{8} \end{matrix} \right]_{k} \left( k+1 \right), \label{RamanujanType2} \\ 
 \frac{16 \sqrt{3} }{3}  & = \sum_{k=0}^{\infty} \left(\frac{1}{4} \right)^{k} \left[ \begin{matrix} \frac{1}{2}, \frac{5}{6}, 
 \frac{7}{6} \vspace{1mm} \\ 1, \frac{4}{3},  \frac{4}{3} \end{matrix} \right]_{k} \left( 18 k + 7 \right), \label{RamanujanType3} \\  
  \frac{91 \sqrt{3}}{16} & =   \sum_{k=0}^{\infty} \left(\frac{1}{4} \right)^{k} \left[ \begin{matrix} \frac{1}{2}, \frac{4}{3}, 
 \frac{5}{3}   \vspace{1mm} \\  1, \frac{19}{12}, \frac{25}{12}  \end{matrix} \right]_{k} (9 k+8), \label{RamanujanType4} \\ 
 6 \sqrt{2} & =  \sum_{k=0}^{\infty} \left(-\frac{1}{4} \right)^{k} \left[ \begin{matrix} 
 \frac{1}{8}, \frac{5}{8}, \frac{3}{4} \vspace{1mm} \\ 
 \frac{7}{8}, 1, \frac{11}{8} 
 \end{matrix} \right]_{k} \left( 40 k+9 \right), \label{RamanujanType5} \\ 
 5 \sqrt{2} & = 
 \sum_{k=0}^{\infty} \left(-\frac{1}{4} \right)^{k} \left[ \begin{matrix} 
 -\frac{1}{8}, \frac{1}{4}, \frac{3}{8} \vspace{1mm} \\ 
 1, \frac{9}{8}, \frac{13}{8} 
 \end{matrix} \right]_{k} (40 k+7), \label{RamanujanType6} \\
 10 \sqrt[3]{2} & = 
 \sum_{k=0}^{\infty} \left(-\frac{1}{4} \right)^{k} \left[ \begin{matrix} 
 \frac{1}{12}, \frac{7}{12}, \frac{2}{3} \vspace{1mm} \\ 
 \frac{11}{12}, 1, \frac{17}{12} 
 \end{matrix} \right]_{k} \left( 60 k+13 \right), \label{RamanujanType7} \\ 
 \frac{512 \sqrt{3}}{81} & = 
 \sum_{k=0}^{\infty} \left(\frac{1}{64} \right)^{k} \left[ \begin{matrix} 
 -\frac{1}{6}, \frac{7}{6}, \frac{3}{2} \vspace{1mm} \\ 
 1, \frac{4}{3}, \frac{5}{3} 
 \end{matrix} \right]_{k} (14 k+11). \label{RamanujanType8}
\end{align}
 Observe that the formulas of Ramanujan type displayed among \eqref{RamanujanType1}--\eqref{RamanujanType8} invole linear polynomial factors 
 within the summands. This may be regarded as ideal or minimal by analogy with the linear polynomial factors involved in Ramanujan's series among 
 \eqref{originalRamanujan1}--\eqref{originalRamanujan3}. In this regard, main results from Chu's work in \cite{Chu2020Infinite} are such that 
 \begin{equation}\label{Ramanujan2over271}
 \frac{3 \sqrt{2} }{ 4 } = 
 \sum_{k=0}^{\infty} \left(\frac{2}{27} \right)^{k} \left[ \begin{matrix} 
 \frac{1}{2}, \frac{1}{2}, \frac{1}{2} \vspace{1mm} \\ 
 \frac{5}{6}, 1, \frac{7}{6} 
 \end{matrix} \right]_{k} \left( 5 k + 1 \right) 
\end{equation}
 and 
 \begin{equation}\label{Ramanujan2over272}
 \frac{3 \pi^2}{4} = 
 \sum_{k=0}^{\infty} \left(\frac{2}{27} \right)^{k} \left[ \begin{matrix} 
 1, 1, 1 \vspace{1mm} \\ 
 \frac{3}{2}, \frac{4}{3}, \frac{5}{3} 
 \end{matrix} \right]_{k} \left( 10 k + 7 \right), 
\end{equation}
 and the formulas in \eqref{Ramanujan2over271}--\eqref{Ramanujan2over272} 
 are described as being of \emph{Ramanujan type} 
 by Chu \cite{Chu2020Infinite}, 
 with each of \eqref{originalRamanujan1}--\eqref{originalRamanujan3} 
 and each of \eqref{RamanujanType1}--\eqref{Ramanujan2over272} being of the form 
 \begin{equation}\label{definitionRamanujanType}
 \mathcal{C} = \sum_{k = 0}^{\infty} z^{k} \left[ \begin{matrix} 
 \alpha, \beta, \gamma \vspace{1mm} \\ 
 A, B, C 
 \end{matrix} \right]_{k} p(k)
\end{equation}
 for a linear polynomial $p(k)$ and a rational or algebraic multiple $\mathcal{C}$ of a fundamental constant such as $\pi$, and where it is not the case that 
 the difference between a lower Pochhammer argument and an upper Pochhammer argument in \eqref{definitionRamanujanType} equals an integer. 
 Informally, series of this form are instances of what are referred to as \emph{Chu-style series} in \cite{Campbell2023,CampbellLevrieunpublished}, with 
 regard to the remarkable discoveries due to Chu reproduced in \eqref{Ramanujan2over271}--\eqref{Ramanujan2over272}. Chu has introduced and 
 proved many further results of a similar nature, including 
\begin{equation}\label{Ramanujan9sqrt3} 
 	 \frac{9 \sqrt{3} }{ 2^{4/3} \pi } = 
 \sum_{n = 0}^{\infty} \left(-\frac{1}{27} \right)^{n} \left[ \begin{matrix} 
 \frac{1}{3}, \frac{2}{3}, \frac{1}{6} \vspace{1mm} \\ 
 1, 1, 1 
 \end{matrix} \right]_{n} \left( 21 n + 2 \right) 
\end{equation}
 and 
\begin{equation}
 \frac{27 \sqrt{3}}{ 2^{5/3} \pi} = 
 \sum_{n = 0}^{\infty} \left(-\frac{1}{27} \right)^{n} \left[ \begin{matrix} 
 \frac{1}{3}, \frac{2}{3}, \frac{5}{6} \vspace{1mm} \\ 
 1, 1, 1 
 \end{matrix} \right]_{n} \left( 42 n + 5 \right), 
\end{equation}
 which are highlighted as main results in \cite{Chu2018}, and Chu also proved 
\begin{equation}
 \pi \sqrt{3} = 
 \sum_{n = 0}^{\infty} \left(\frac{1}{9} \right)^{n} \left[ \begin{matrix} 
 1, \frac{3}{4}, \frac{5}{4} \vspace{1mm} \\ 
 \frac{3}{2}, \frac{3}{2}, \frac{3}{2} 
 \end{matrix} \right]_{n} \left( 8 n + 5 \right), 
\end{equation}
 and 
\begin{equation}\label{Ramanujan3pisquared}
 \frac{3\pi^2}{4} = 
 \sum_{n = 0}^{\infty} \left(\frac{2}{27} \right)^{n} \left[ \begin{matrix} 
 1, 1, 1 \vspace{1mm} \\ 
 \frac{3}{2}, \frac{4}{3}, \frac{5}{3} 
 \end{matrix} \right]_{n} \left( 10 n + 7 \right), 
\end{equation}
  with \eqref{Ramanujan9sqrt3}--\eqref{Ramanujan3pisquared}   highlighted as main results in \cite{Chu2021Infinite}.   The new series of convergence rate  
  $\frac{1}{4}$ displayed among   \eqref{RamanujanType1}--\eqref{RamanujanType8}   are also reminiscent of Guillera's formula   
\begin{equation}\label{Guilleraquarter}
 \frac{\pi^2}{4} = \sum_{n = 0}^{\infty} \left(\frac{1}{4} \right)^{n} \left[ \begin{matrix} 
 1, 1, 1 \vspace{1mm} \\ 
 \frac{3}{2}, \frac{3}{2}, \frac{3}{2} 
 \end{matrix} \right]_{k} \left( 3 n + 2 \right) 
\end{equation}
    which was introduced and proved via the WZ method by Guillera     in 2008 \cite{Guillera2008} and reproved in the Chu--Zhang    article \cite{ChuZhang2014}   
  that has inpired our current work.   

 Our method can also be applied to prove the formula of Ramanujan type such that 
\begin{equation}\label{onlyequivalent}
 \frac{15 \sqrt{2} }{4} = 
 \sum_{k=0}^{\infty} \left(-\frac{1}{27} \right)^{k} \left[ \begin{matrix} 
 -\frac{1}{4}, \frac{5}{4}, \frac{5}{4} \vspace{1mm} \\ 
 1, \frac{13}{12}, \frac{17}{12} 
 \end{matrix} \right]_{k} (28 k+5), 
\end{equation}
 and this can be seen as an improved version, with regard to the linear polynomial in \eqref{onlyequivalent},     of the equivalent formula   
\begin{equation}\label{quadraticequiv}
 \frac{15 \sqrt{2}}{2} = 
 \sum_{k=0}^{\infty} \left(-\frac{1}{27} \right)^{k} \left[ \begin{matrix} 
 \frac{1}{4}, \frac{1}{4}, \frac{3}{4} \vspace{1mm} \\ 
 1, \frac{13}{12}, \frac{17}{12} 
 \end{matrix} \right]_{k} (224 k^2 + 124 k + 11), 
\end{equation}
 given by Chu in both \cite{Chu2023} and \cite{Chu2021Infinite}. 
 We find that \eqref{onlyequivalent} and \eqref{quadraticequiv}
 are equivalent up to a series such that its partial sums admit a closed form, 
 and, informally, this is due to how the rising factorials involved in 
 \eqref{onlyequivalent} and \eqref{quadraticequiv}
 are equivalent up to integer differences. 
 By way of contrast, 
 our new series highlighted in the Examples in 
 Section \ref{sectionmain} 
 are not equivalent to previously known series, again with reference to 
 the work of Chu et al.\ 
 \cite{Chu2023,Chu2018,Chu2020Alternating,Chu2021Further,Chu2020Infinite,Chu2021Infinite,Chu2021Ramanujan,Chu2019,ChuZhang2014}. 

 In addition to the new formulas of Ramanujan type highlighted among 
 \eqref{RamanujanType1}--\eqref{RamanujanType8}, 
 we list the following results that we prove through our Zeilberger-based method 
 as motivating results. 
 To the best of our knowledge, each of the following $\pi$ formulas 
 has not previously been known: 
\begin{align}
 6 \pi ^2 & = 
 \sum_{k=0}^{\infty} \left(\frac{1}{64} \right)^{k} \left[ \begin{matrix} 
 1, 1 \vspace{1mm} \\ 
 \frac{5}{4}, \frac{7}{4} 
 \end{matrix} \right]_{k} \frac{189 k^2+297 k+118}{(3 k+1) (3 k+2)}, \label{PiMotivating1} \\
 \frac{2 \pi }{\sqrt{3}} & = 
 \sum_{k=0}^{\infty} \left(\frac{1}{4} \right)^{k} \left[ \begin{matrix} 
 \frac{1}{3} \vspace{1mm} \\ 
 \frac{7}{6} 
 \end{matrix} \right]_{k} \frac{9 k+7}{(k+1) (3 k+2)}, \label{PiMotivating2} \\
 \frac{4 \pi }{\sqrt{3}} & = 
 \sum_{k=0}^{\infty} \left(-\frac{1}{4} \right)^{k} \left[ \begin{matrix} 
 \frac{4}{3} \vspace{1mm} \\ 
 \frac{7}{6} 
 \end{matrix} \right]_{k} \frac{30 k^2+45 k+16}{(k+1) (2 k+1) (3 k+2)}, \label{PiMotivating3} \\ 
 \frac{315 \pi }{16 \sqrt{2}} & = 
 \sum_{k=0}^{\infty} \left(-\frac{1}{27} \right)^{k} \left[ \begin{matrix} 
 \frac{1}{8}, \frac{1}{4}, \frac{5}{8}, 1 \vspace{1mm} \\ 
 \frac{13}{12}, \frac{11}{8}, \frac{17}{12}, \frac{15}{8} 
 \end{matrix} \right]_{k} \text{\footnotesize{$ \left( 448 k^3+688 k^2+305 k+44 \right)$,}} \label{PiMotivating4} \\ 
 \frac{5 \pi }{2} & = 
 \sum_{k=0}^{\infty} \left(\frac{1}{4} \right)^{k} \left[ \begin{matrix} 
 \frac{1}{4}, \frac{1}{2} \vspace{1mm} \\ 
 \frac{9}{8}, \frac{13}{8} 
 \end{matrix} \right]_{k} \frac{48 k^2+68 k+23}{(k+1) (4 k+3)}, \label{PiMotivating5} \\
 \frac{45 \pi }{4} & = 
 \sum_{k=0}^{\infty} \left(\frac{4}{27} \right)^{k} \left[ \begin{matrix} 
 1, \frac{1}{2}, \frac{1}{2}, \frac{1}{2} \vspace{1mm} \\ 
 \frac{7}{6}, \frac{5}{4}, \frac{7}{4}, \frac{11}{6} 
 \end{matrix} \right]_{k} \text{\footnotesize{$\left( 184 k^3+338 k^2+193 k+32 \right)$,}} \label{PiMotivating6} \\
 \frac{165 \pi}{4 \sqrt{3}} & = 
 \sum_{k=0}^{\infty} \left(\frac{1}{4} \right)^{k} \left[ \begin{matrix} 
 \frac{5}{6}, \frac{5}{3} \vspace{1mm} \\ 
 \frac{17}{12}, \frac{23}{12} 
 \end{matrix} \right]_{k} \frac{108 k^2+174 k+71}{(k+1) (6 k+1)}, \label{PiMotivating7} \\ 
 \frac{15 \pi }{ 4 \sqrt{3} } & = 
 \sum_{k=0}^{\infty} \left(-\frac{1}{27} \right)^{k} \left[ \begin{matrix} 
 \frac{1}{2}, \frac{2}{3}, 1 \vspace{1mm} \\ 
 \frac{4}{3}, \frac{4}{3}, \frac{11}{6} 
 \end{matrix} \right]_{k} \left( 21 k^2+26 k+7 \right), \label{PiMotivating8} \\
 \frac{9 \pi ^2}{4} & = 
 \sum_{k=0}^{\infty} \left(\frac{1}{64} \right)^{k} \left[ \begin{matrix} 
 \frac{1}{2}, 1, 1, 1 \vspace{1mm} \\ 
 \frac{5}{4}, \frac{5}{4}, \frac{7}{4}, \frac{7}{4} 
 \end{matrix} \right]_{k} \left( 42 k^2+61 k+22 \right). \label{PiMotivating9}
\end{align} 

\section{Hypergeometric accelerations with shifted indices}\label{sectionmain}
 Building on an acceleration method that was given by Amdeberhan and Zeilberger in 1997 \cite{AmdeberhanZeilberger1997} and that relies on the WZ 
 method, Wilf considered something of a different approach, by applying \emph{Zeilberger's algorithm}, as opposed to the WZ method, to obtain accelerations 
 from \emph{inhomogeneous} and first-order difference equations for infinite hypergeometric series, in contrast to series identities 
 obtained from the difference equation $F(n+1,k) - F(n, 
 k) = G(n,k+1) - G(n,k)$ required according to the WZ method. In particular, in the exceptional cases whereby it is possible to obtain a 
 \emph{first-order} recurrence of the form 
\begin{equation}\label{firstorder}
 p_{1}(n) F(n+r, k) + p_{2}(n) F(n, k) = G(n, k+1) - G(n,k) 
\end{equation}
 through the application of Zeilberger's algorithm 
 to a bivariate hypergeometric function $F(n, k)$, 
 for polynomials $p_{1}(n)$ and $p_{2}(n)$, 
 the difference equation in \eqref{firstorder} can 
 be applied, subject to a vanishing condition indicated below, 
 to obtain series accelerations via the telescoping of the right-hand side of 
 \eqref{firstorder}, upon summation with respect to $k$ \cite{CampbellLevrieunpublished,Wilf1999}. 
 Writing $\mathcal{F}(n) = \sum_{k=0}^{\infty} F(n, k)$, we obtain 
\begin{equation}\label{20294909590939390909A4444444494M494A}
 p_{1}(n) \mathcal{F}(n+r) + p_2(n) \mathcal{F}(n) = -G(n, 0), 
\end{equation}
 under the assumption that the series in \eqref{20294909590939390909A4444444494M494A}
 are convergent, yielding 
\begin{equation}\label{20828480858083828252827282A2M2A}
 \mathcal{F}(n) = \frac{-G(n, 0)}{p_{2}(n)} - \frac{p_1(n)}{p_2(n)} \mathcal{F}(n+r). 
\end{equation} 
 Wilf's acceleration method \cite{Wilf1999} is given by the iterative application of the 
 recursion in \eqref{20828480858083828252827282A2M2A}, 
 for cases whereby the leading coefficients of the polynomials $p_{1}(n)$ 
 and $p_{2}(n)$ 
 provide an accelerated rate of convergence given by products of consecutive 
 expressions of the form $-\frac{p_{1}(n)}{p_{2}(n)}$, as clarified in Sections 
 \ref{208284808580838382848A8M1A} and \ref{202408888888580838781828AM1A} below. 

 Wilf applied the above described acceleration method with the use of the input functions 
 $F(n, k) = \frac{(k!)^2}{((n+k+1)!)^2}$ 
 and $F(n, k) = \frac{ (k!)^3 }{ ((k+2n+1)!)^3 }$, 
 yielding families, according to the free parameter $n$, 
 of accelerated series for $\pi^2$ and Ap\'{e}ry's constant $\zeta(3) 
 = \frac{1}{1^3} + \frac{1}{2^3} + \cdots$, respectively. 
 Our recent paper \cite{CampbellLevrieunpublished} may be seen as taking Wilf's ideas a step forward 
 through the application of 
 first-order recurrences for \emph{multiparameter} hypergeometric functions $F(n, k)$ 
 such as 
\begin{equation}\label{previousquarter}
 F(n, k) = \frac{(a)_{k} (b)_{k} }{ (n)_{k}^{2} } 
\end{equation}
 and 
\begin{equation}\label{previousfor4over27}
 F(n, k) = \frac{ (a)_{k} (b)_{k} }{ (n)_{k} (2n)_{k} } 
\end{equation}
 for free parameters $a$ and $b$. The key idea in this paper is to generalize accelerations of the form considered in 
 \cite{CampbellLevrieunpublished} through the use of shifted indices for Pochhammer symbols. 
 More specifically, given 
 a hypergeometric function $F(n, k)$ that yields series accelerations according to Wilf's method 
 \cite{Wilf1999}, by expressing $F(n, k)$ in terms of Pochhammer symbols, 
 if we instead apply Wilf's method with respect to the function obtained
 by shifting a Pochhammer symbol such as $( \mathcal{G} )_{k}$, 
 for a combination $\mathcal{G}$ of parameters and $n$- and $k$-expressions, 
 by replacing it with $( \mathcal{G} )_{k + p}$ for a free parameter $p$, 
 this provides a way of improving upon 
 families of accelerations obtained via Wilf's method in \cite{CampbellLevrieunpublished,Wilf1999}. 

 A computational difficulty that arises via our ``shifted'' variant of Wilf's method 
 has to do with searching for closed-form constants
 arising from canellations or simplifications of combinations of $\Gamma$-expressions 
 arising from special values of the left-hand side of 
 \eqref{20828480858083828252827282A2M2A} together with the accelerated summand. 
 However, exhaustive computer searches through 
 combinations of rational parameter values involved 
 in shifted variants of input functions as in \eqref{previousquarter}--\eqref{previousfor4over27}
 allow us to demonstrate the extent to which 
 our shifted method builds upon 
 the work of Chu and Zhang \cite{ChuZhang2014} 
 and related works on Chu-style series as in \cite{Chu2021Infinite,Chu2021Ramanujan}. 
 We begin by shifting the indices for each of the Pochhammer factors involved in \eqref{previousquarter}, 
 yielding a hypergeometric transform that involves seven free parameters and that 
 has led us to experimentally discover and prove 
 new results as in the formulas of Ramanujan type highlighted 
 among \eqref{RamanujanType1}--\eqref{RamanujanType4}, 
 and to recover many of the results due to Chu and Zhang \cite{ChuZhang2014}. 

 Since our paper is largely motivated by Chu and Zhang's accelerated $\pi$ formulas 
 \cite{ChuZhang2014}, the remainder of our paper is organized in 
 such a way so as to agree with the organization of \cite{ChuZhang2014}. 
 In particular, we restrict our attention, for the time being, 
 to series that are of the same convergence rates considered by Chu and Zhang in \cite{ChuZhang2014}. 
 Similarly, the mathematical constants considered in this paper, 
 as in the closed forms we obtain for accelerated series, 
 agree with the constants considered in 
 \cite{ChuZhang2014} and \cite{CampbellLevrieunpublished}. 

\subsection{Series of convergence rate $\frac{1}{4}$}\label{208284808580838382848A8M1A}
 Our first family of series accelerations is derived through 
 our application of Zeilberger's algorithm to 
\begin{equation}\label{mainFforquarter}
 F(n, k) := \frac{(a)_{k+f} (b)_{k+e}}{ (n)_{k+d} (n)_{k+c}} 
\end{equation}
 for free parameters $a$, $b$, $c$, $d$, $e$, and $f$. 
 As in our recent work \cite{CampbellLevrieunpublished} on Wilf's acceleration method, our 
 discoveries based on Zeilberger's algorithm were obtained through the experimental use of the Maple
 Computer Algebra System, with the use of the Maple {\tt SumTools} package to be loaded with the following input. 
\begin{verbatim}
with(SumTools[Hypergeometric]):
\end{verbatim}
 As in \cite{CampbellLevrieunpublished}, our experimental discoveries 
 mainly rely on the $r = 1$ case of \eqref{20294909590939390909A4444444494M494A}. 
 For this $r = 1$ case, 
 we rewrite the recursion in \eqref{20828480858083828252827282A2M2A} as 
\begin{equation}\label{mathcalFr1}
 \mathcal{F}(n) = g_{1}(n) + g_{2}(n) \mathcal{F}(n+1), 
\end{equation}
 for rational functions 
\begin{equation}\label{20245050341AMA}
 g_{1}(n) = \frac{-G(n, 0)}{p_{2}(n)} \ \ \ 
 \text{and} \ \ \ g_{2}(n) = - \frac{p_{1}(n)}{p_{2}(n)}, 
\end{equation}
 so that the repeated application 
 of the hypergeometric recurrence relation in \eqref{mathcalFr1} yields 
\begin{equation}\label{mainacceleration}
 \mathcal{F}(n) = \sum_{j=0}^{\infty} \left( \prod_{i=0}^{j-1} g_{2}(n+i) \right) g_{1}(n + j) 
\end{equation}
 if 
\begin{equation}\label{desiredvanishing}
 \lim_{m \to \infty} \prod_{o=0}^{m} g_{2}(n+i) \mathcal{F}(n+m+1) 
\end{equation}
 vanishes \cite{CampbellLevrieunpublished}. 
 With regard to the index-shifted function in \eqref{mainFforquarter}, we input 
\begin{verbatim}
T := pochhammer(a, k + f)*pochhammer(b, k + e)/(pochhammer(n, 
k + d)*pochhammer(n, k + c))
\end{verbatim}
 so that the output from 
\begin{verbatim}
Zpair := Zeilberger(T, n, k, En):
\end{verbatim}
 provides the two-term recursion 
 required according to our acceleration based on \eqref{mainFforquarter}, 
 with the following rational certificate: 

 \ 

\noindent $ R(n, k) = -n^2 \big(a b-a c-a d+a e-a k-2 a n+a-b c-b d+b f-b k- 
 2 b n+b+c^2+c d-c e-c f+c k+3 c n-c+d^2-d e-d f+d k+3 d n-d + 
 e f-e k-2 e n+e-f k-2 f n+f+2 k n+3 n^2-2 n\big)$.

\begin{theorem}\label{theoremquarter}
 For $F(n, k)$ as in \eqref{mainFforquarter} and for the corresponding rational certificate $R(n, k)$ given above, 
 we write $G(n, k) = R(n, k) F(n, k)$. 
 For $\mathcal{F}(n) = \sum_{k=0}^{\infty} F(n, k)$, 
 the recursion in \eqref{mathcalFr1} is satisfied 
 for the rational functions in \eqref{20245050341AMA} and for the following polynomials: 

 \ 

\noindent $ p_{1}(n) = -a^2 b^2 + a^2 b c + a^2 b d - 2 a^2 b e + 2 a^2 b n - a^2 c d + 
 a^2 c e - a^2 c n + a^2 d e - a^2 d n - a^2 e^2 + 2 a^2 e n - 
 a^2 n^2 + a b^2 c + a b^2 d - 2 a b^2 f + 2 a b^2 n - a b c^2 - 
 2 a b c d + 2 a b c e + 2 a b c f - 4 a b c n - a b d^2 + 
 2 a b d e + 2 a b d f - 4 a b d n - 4 a b e f + 4 a b e n + 
 4 a b f n - 4 a b n^2 + a c^2 d - a c^2 e + a c^2 n + a c d^2 - 
 2 a c d e - 2 a c d f + 4 a c d n + a c e^2 + 2 a c e f - 
 4 a c e n - 2 a c f n + 3 a c n^2 - a d^2 e + a d^2 n + a d e^2 + 
 2 a d e f - 4 a d e n - 2 a d f n + 3 a d n^2 - 2 a e^2 f + 
 2 a e^2 n + 4 a e f n - 4 a e n^2 - 2 a f n^2 + 2 a n^3 - b^2 c d + 
 b^2 c f - b^2 c n + b^2 d f - b^2 d n - b^2 f^2 + 2 b^2 f n - 
 b^2 n^2 + b c^2 d - b c^2 f + b c^2 n + b c d^2 - 2 b c d e - 
 2 b c d f + 4 b c d n + 2 b c e f - 2 b c e n + b c f^2 - 
 4 b c f n + 3 b c n^2 - b d^2 f + b d^2 n + 2 b d e f - 2 b d e n + 
 b d f^2 - 4 b d f n + 3 b d n^2 - 2 b e f^2 + 4 b e f n - 
 2 b e n^2 + 2 b f^2 n - 4 b f n^2 + 2 b n^3 - c^2 d^2 + c^2 d e + 
 c^2 d f - 2 c^2 d n - c^2 e f + c^2 e n + c^2 f n - c^2 n^2 + 
 c d^2 e + c d^2 f - 2 c d^2 n - c d e^2 - 2 c d e f + 4 c d e n - 
 c d f^2 + 4 c d f n - 4 c d n^2 + c e^2 f - c e^2 n + c e f^2 - 
 4 c e f n + 3 c e n^2 - c f^2 n + 3 c f n^2 - 2 c n^3 - d^2 e f + 
 d^2 e n + d^2 f n - d^2 n^2 + d e^2 f - d e^2 n + d e f^2 - 
 4 d e f n + 3 d e n^2 - d f^2 n + 3 d f n^2 - 2 d n^3 - e^2 f^2 + 
 2 e^2 f n - e^2 n^2 + 2 e f^2 n - 4 e f n^2 + 2 e n^3 - f^2 n^2 + 
 2 f n^3 - n^4$, 

 \ 

\noindent $ p_{2}(n) = a^2 n^2 - 4 a n^3 + b^2 n^2 - 4 b n^3 + c^2 n^2 + 4 c n^3 + d^2 n^2 + 4 d n^3 + e^2 n^2 - 
 4 e n^3 + f^2 n^2 - 4 f n^3 - 2 n^3 + 4 n^4 + 2 a b n^2 - 2 a c n^2 - 2 a d n^2 + 2 a e n^2 + 
 2 a f n^2 - 2 b c n^2 - 2 b d n^2 + 2 b e n^2 + 2 b f n^2 + 2 c d n^2 - 2 c e n^2 - 2 c f n^2 - 
 2 d e n^2 - 2 d f n^2 + 2 e f n^2 + a n^2 + b n^2 - c n^2 - d n^2 + e n^2 + f n^2$. 
\end{theorem}

\begin{proof}
 For $p_{1}(n)$ and $p_{2}(n)$ as specified, we may verify that the difference equation 
 $ p_{1}(n) F(n+1, k) + p_{2}(n) F(n, k) = G(n, k+1) - G(n, k)$ holds, 
 so that the telescoping of the right-hand side, upon being subjected to summation with respect to $k$, 
 yields the desired result. 
\end{proof}

  It is a matter of routine to verify that the desired vanishing of \eqref{desiredvanishing}  holds. So, from the recursion in Theorem \ref{theoremquarter},  
  we find that \eqref{mainacceleration}    gives us a hypergeometric transform,   according to the values of $g_{1}(n)$ and $g_{2}(n)$ 
 determined by $p_{1}(n)$ and $p_{2}(n)$. 
 In particular, the limiting behaviour of $g_{2}(n)$ 
 gives us that \eqref{mainacceleration}
 accelerates the convergence rate of $\mathcal{F}(n)$ from $1$ 
 to the convergence rate of $\frac{1}{4}$ given by the right-hand side of 
 \eqref{mainacceleration}. 
 Despite the unwieldy appearance of the required polynomials $p_{1}(n)$ and $p_{2}(n)$, 
 specifying rational values for the free parameters $a$, $b$, $c$, $d$, $e$, $f$, and $n$ in 
 \eqref{mainacceleration} 
 produces remarkably simple summands and closed forms, as below. 

  The accelerated series in Examples \ref{20240504832PM1A}--\ref{20240504833PM2222A}  below are derived from  
  Theorem \ref{theoremquarter}   and the acceleration formula in  
 \eqref{mainacceleration},   according to the specified parameter values in Examples \ref{20240504832PM1A}--\ref{20240504833PM2222A}. 
 The accelerated series that are 
 highlighted with displayed formulas in the Examples throughout this 
 paper are new and are not equivalent (including up to reindexings and integer differences of Pochhammer arguments) 
 to previously published formulas, to the best of our knowledge. 

\begin{example}\label{20240504832PM1A}
 Setting $(a, b, c, d, e, f, n) = \big( \frac{1}{3}, \frac{1}{3}, 1, \frac{1}{3}, \frac{1}{3}, \frac{2}{3}, 1 \big)$, we obtain $$ \frac{10 \pi }{\sqrt{3}} = 
 \sum_{j=0}^{\infty} \left(\frac{1}{4} \right)^{j} \left[ \begin{matrix} 
 \frac{2}{3} \vspace{1mm} \\ 
 \frac{11}{6}
 \end{matrix} \right]_{j} \frac{27 j^2+42 j+17}{(j+1) (3 j+1)}. $$ 
\end{example}

\begin{example}
 Setting $(a, b, c, d, e, f, n) = 
 \big( \frac{1}{6}, \frac{1}{6}, \frac{1}{6}, 1, \frac{2}{3}, \frac{2}{3}, \frac{5}{6} \big)$, we obtain 
 $$ \frac{35 \pi }{18} = 
 \sum_{j=0}^{\infty} \left(\frac{1}{4} \right)^{j} \left[ \begin{matrix} 
 \frac{1}{6}, \frac{1}{6}, 1 \vspace{1mm} \\ 
 \frac{13}{12}, \frac{19}{12}, \frac{11}{6} 
 \end{matrix} \right]_{j} \left( 18 j^2+19 j+6 \right). $$ 
\end{example}

\begin{example}
 Setting $(a, b, c, d, e, f, n) = 
 \big( \frac{1}{2}, \frac{1}{4}, \frac{3}{4}, \frac{3}{2}, \frac{3}{4}, 0, \frac{1}{2} \big)$, we obtain 
 $$ \frac{21 \pi }{2} = 
 \sum_{j=0}^{\infty} \left(\frac{1}{4} \right)^{j} \left[ \begin{matrix} 
 \frac{3}{4}, \frac{3}{2} \vspace{1mm} \\ 
 \frac{11}{8}, \frac{15}{8}
 \end{matrix} \right]_{j} \frac{48 j^2+76 j+31}{(j+1) (4 j+1)}. $$ 
\end{example}

\begin{example}\label{exsimilarrising}
 Setting $(a, b, c, d, e, f, n) = \big( \frac{1}{4}, \frac{1}{4}, \frac{1}{4}, 0, 0, 0, \frac{3}{4} \big)$, we obtain 
 $$ \frac{5 \sqrt{2}}{4} = 
 \sum_{j=0}^{\infty} \left(\frac{1}{4} \right)^{j} \left[ \begin{matrix} 
 \frac{1}{2}, \frac{1}{2}, \frac{3}{4} \vspace{1mm} \\ 
 1, \frac{9}{8}, \frac{13}{8} 
 \end{matrix} \right]_{j}  \left(12 j^2+9 j + 1 \right). $$ 
\end{example}

\begin{example}
 Setting $(a,b,c,d,e, f, n) = \big( \frac{3}{8}, \frac{5}{8}, \frac{1}{8}, \frac{7}{8}, 0, 0, 1 \big)$, 
 we obtain 
 $$ 14 \sqrt{2} = 
 \sum_{j=0}^{\infty} \left(\frac{1}{4} \right)^{j} \left[ \begin{matrix} 
 \frac{1}{2}, \frac{3}{4}, \frac{5}{4} \vspace{1mm} \\ 
 \frac{9}{8}, \frac{15}{8}, 2 
 \end{matrix} \right]_{j} \left( 24 j^2+40 j+17 \right), $$ 
 noting that this is inequivalent to Example \ref{exsimilarrising}, 
 noting the respective values of the upper Pochhammer symbols. 
\end{example}

\begin{example}
 Setting $(a,b,c, d, e, f, n) = \big( \frac{1}{12}, \frac{5}{12}, \frac{1}{4}, 0, 0, 0, 1 \big)$, we obtain 
 $$ \frac{567 \sqrt{2}}{16} = 
 \sum_{j=0}^{\infty} \left(\frac{1}{4} \right)^{j} \left[ \begin{matrix} 
 \frac{7}{12}, \frac{5}{6}, \frac{11}{12}, \frac{7}{6} \vspace{1mm} \\ 
 1, \frac{5}{4}, \frac{11}{8}, \frac{15}{8} 
 \end{matrix} \right]_{j} \left(108 j^2+135 j+35\right). $$ 
\end{example}

\begin{example}
 Setting $(a,b,c, d, e, f, n) = \big( \frac{1}{12}, \frac{7}{12}, \frac{1}{6}, 0, 0, 0, \frac{5}{6} \big)$, we obtain 
 $$ \frac{28 \sqrt{2}}{3} = 
 \sum_{j=0}^{\infty} \left(\frac{1}{4} \right)^{j} \left[ \begin{matrix} 
 \frac{1}{4}, \frac{5}{12}, \frac{3}{4}, \frac{11}{12} \vspace{1mm} \\ 
 \frac{5}{6}, 1, \frac{13}{12}, \frac{19}{12} 
 \end{matrix} \right]_{j} \left(144 j^2+104 j+9\right). $$ 
\end{example}

\begin{example}
 Setting $(a,b,c,d, e, f, n) = \big( \frac{1}{6}, \frac{5}{6}, \frac{1}{2}, 0, 0, 0, 1 \big)$, we obtain 
 $$ \frac{81 \sqrt{3}}{8} = 
 \sum_{j=0}^{\infty} \left(\frac{1}{4} \right)^{j} \left[ \begin{matrix} 
 \frac{1}{6}, \frac{2}{3}, \frac{5}{6}, \frac{4}{3} \vspace{1mm} \\ 
 1, \frac{5}{4}, \frac{3}{2}, \frac{7}{4} 
 \end{matrix} \right]_{j} \left(54 j^2 + 63 j + 16\right). $$ 
\end{example}

\begin{example}
 Setting $(a,b,c, d, e, f, n) = \big( \frac{1}{6}, \frac{1}{6}, \frac{1}{3}, 0, 0, 0, \frac{2}{3} \big)$, we obtain 
 $$ \frac{32 \sqrt{3}}{9} = 
 \sum_{j=0}^{\infty} \left(\frac{1}{4} \right)^{j} \left[ \begin{matrix} 
 \frac{1}{2}, \frac{1}{2}, \frac{5}{6}, \frac{5}{6} \vspace{1mm} \\ 
 \frac{2}{3}, 1, \frac{7}{6}, \frac{5}{3} 
 \end{matrix} \right]_{j}  \left(36 j^2+28 j+3\right). $$ 
\end{example}

\begin{example}
 Setting $(a, b, c, d, e, f, n) = \big( \frac{1}{6}, \frac{1}{2}, \frac{2}{3}, \frac{5}{6}, 0, 0, \frac{1}{6} \big)$, we obtain 
 $$ \frac{7 \sqrt{3}}{2} = 
 \sum_{j=0}^{\infty} \left(\frac{1}{4} \right)^{j} \left[ \begin{matrix} 
 \frac{1}{3}, \frac{1}{2}, \frac{2}{3} \vspace{1mm} \\ 
 1, \frac{13}{12}, \frac{19}{12}
 \end{matrix} \right]_{j} \left( 54 j^2+39 j+4 \right). $$ 
\end{example}

\begin{example}
 Setting $(a,b,c,d, e, f, n) = \big( \frac{1}{6}, \frac{2}{3}, -\frac{2}{3}, \frac{2}{3}, 0, 0, \frac{5}{3} \big)$, we obtain 
 $$ \frac{405 \sqrt[3]{2}}{7} = 
 \sum_{j=0}^{\infty} \left(\frac{1}{4} \right)^{j} \left[ \begin{matrix} 
 \frac{1}{3}, \frac{5}{6}, \frac{5}{3}, \frac{13}{6} \vspace{1mm} \\ 
 1, \frac{7}{4}, \frac{9}{4}, \frac{7}{3} 
 \end{matrix} \right]_{j} \left( 54 j^2+114 j+65 \right). $$ 
\end{example}

\begin{example}
 Setting $(a,b,c,d, e, f, n) = \big( \frac{1}{6}, \frac{5}{6}, \frac{1}{6}, \frac{5}{6}, 0, 0, 1 \big)$, we obtain the formula of Ramanujan type 
 highlighted in \eqref{RamanujanType1}. 
\end{example}

\begin{example}
 Setting $(a, b, c, d, e, f, n) = \big( \frac{1}{4}, \frac{1}{4}, \frac{1}{6}, \frac{11}{12}, 0, 0, \frac{5}{6} \big)$, we obtain the formula of Ramanujan 
 type highlighted in \eqref{RamanujanType2}. 
\end{example}

\begin{example}
 Setting $(a, b, c, d, e, f, n) = \big( \frac{1}{6}, \frac{1}{6}, \frac{1}{3}, \frac{2}{3}, \frac{1}{3}, 0, \frac{2}{3} \big)$, we obtain the formula of 
 Ramanujan type highlighted in \eqref{RamanujanType3}. 
\end{example}

\begin{example}
 Setting $(a,b,c, d, e, f, n) = \big( \frac{1}{6}, \frac{1}{2}, -\frac{5}{6}, 0, 0, 0, \frac{11}{6} \big)$, we obtain the formula of Ramanujan type 
 highlighted in \eqref{RamanujanType4}. 
\end{example}

\begin{example}
 Setting $(a,b,c, d, e, f, n) = \big( \frac{3}{2}, 1, \frac{7}{4}, 0, 0, 0, \frac{1}{4} \big)$, we obtain 
 the new and motivating $\pi$ formula highlighted in \eqref{PiMotivating5}. 
\end{example}

\begin{example}
 Setting $(a, b, c, d, e, f, n) = 
 \big( 1, \frac{1}{6}, \frac{1}{2}, \frac{4}{3}, \frac{1}{6}, 0, \frac{2}{3} \big)$, we obtain 
 the new and motivating $\pi$ formula highlighted in \eqref{PiMotivating7}. 
\end{example}

\begin{example}\label{20240504833PM2222A}
 Setting $(a, b, c, d, e, f, n) = \big( 1, \frac{2}{3}, \frac{1}{3}, \frac{2}{3}, \frac{2}{3}, 0, \frac{4}{3} \big)$, we obtain the new and motivating $\pi$ 
 formula highlighted in \eqref{PiMotivating2}. 
\end{example}

 Setting $(a,b,c, d, e, f, n)$ to be among the tuples 
 $\big( \frac{1}{3}$, $ \frac{1}{3}$, $ \frac{1}{3}$, $ 0$, $ 0$, $ 0$, $ 1 \big)$, 
 $\big( \frac{2}{3}$, $ \frac{2}{3}$, $ \frac{2}{3}$, $ 0$, $ 0$, $ 0$, $ 1 \big)$, 
 $\big( \frac{1}{2}$, $ \frac{1}{2}$, $ \frac{1}{2}$, $ \frac{1}{2}$, $ 0$, $ 0$, $ 1 \big)$, 
 $ \big( \frac{1}{2}$, $ 1$, $ \frac{1}{2}$, $ \frac{1}{2}$, $ 0$, $ 0$, $ 1 \big)$, and 
 $\big( \frac{1}{4}$, $ \frac{1}{4}$, $ 0$, $ 1$, $ \frac{1}{4}$, $ \frac{1}{4}$, $ 1 \big)$, 
 we recover, in equivalent ways, distinct results from Chu and Zhang's work on accelerating Dougall's sum \cite{ChuZhang2014}, including 
 Guillera's formula in \eqref{Guilleraquarter} 
 \cite{Guillera2008} and Ramanujan's formula in \eqref{originalRamanujan1} \cite{Ramanujan1914}. 
 For the sake of brevity, we leave the details to the reader, 
 and emphasize new results provided in this paper. 

 Setting $(a,b,c, d, e, f, n)$ to be among the tuples 
 $ \big( \frac{3}{4}$, $ \frac{3}{4}$, $ \frac{3}{4}$, $ 0$, $ 0$, $ 0$, $ 1 \big)$, 
 $ \big( \frac{1}{4}$, $ \frac{1}{4}$, $ -\frac{1}{4}$, $ 0$, $ 0$, $ 0$, $ \frac{5}{4} \big)$, 
 $ \big( \frac{3}{4}$, $ \frac{3}{4}$, $ -\frac{3}{4}$, $ 0$, $ 0$, $ 0$, $ \frac{7}{4} \big)$, 
 $ \big( \frac{2}{3}$, $ \frac{2}{3}$, $ -\frac{2}{3}$, $ 0$, $ 0$, $ 0$, $ \frac{5}{3} \big)$, and 
 $ \big( \frac{1}{3}$, $ \frac{1}{6}$, $ \frac{5}{6}$, $ \frac{1}{2}$, $ \frac{1}{6}$, $ 0$, $ \frac{1}{2} \big)$, 
 we recover distinct results from 
 from Chu's work on $\Omega$-sums \cite{Chu2021Infinite}. 

\subsection{Series of convergence rate $-\frac{1}{4}$}\label{202408888888580838781828AM1A}
 The alternating series due to Chu and Zhang \cite{ChuZhang2014} 
 of absolute convergence rate $\frac{1}{4}$
 lead us to consider how it may be possible to derive alternating analogues of the results from the preceding subsection. 
 In this regard, series of convergence rate $-\frac{1}{4}$ were derived in 
 our past work \cite{CampbellLevrieunpublished} by applying Zeilberger's algorithm
 to an equivalent version of 
\begin{equation}\label{2072470757073767477AM1A}
 F(n, k) = \frac{(-1)^k (a)_{k} (b)_{k}}{(a+n)_k (b+n)_k}. 
\end{equation}
 This leads us to obtain accelerations via index shifts applied to 
 \eqref{2072470757073767477AM1A}, by analogy with \eqref{mainFforquarter}. 
 However, this proves to be problematic in the sense that 
 shifted variants of \eqref{2072470757073767477AM1A} such as 
 $$ \frac{(-1)^k (a)_{k} (b)_{k}}{(a+n)_k (b+n)_{k+c}} $$
 and 
 $$ \frac{(-1)^k (a)_{k} (b)_{k}}{(a+n)_{k+d} (b+n)_{k+c}}, $$
 for free parameters $c$ and $d$, 
 do not admit first-order recurrences according to Zeilberger's algorithm.
 This illustrates the difficulty associated with 
 finding hypergeometric functions $F(n, k)$ that 
 admit two-term recursions in such a way so that 
 \eqref{mainacceleration} yields series accelerations, again 
 subject to the vanishing of 
 \eqref{desiredvanishing}. 
 For the purposes of our introducing series of convergence rate $-\frac{1}{4}$, 
 the multiparameter input function 
\begin{equation}\label{Fnegativequarter}
 F(n, k) := \frac{(-1)^k (a)_{k+c} (b)_{k+c}}{(a+n)_k (b+n)_k} 
\end{equation}
 suffices. In this case, for Zeilberger's algorithm applied to \eqref{Fnegativequarter}, we obtain the $r = 2$ case of \eqref{firstorder}, 
 as opposed to the $r = 1$ case applied in the preceding subsection 
 and in following subsections. 
 This seems to point toward the ``black box'' nature associated with WZ-type methods, 
 as expressed by Guillera \cite{Guillera2018Dougall}
 in the context of how Zeilberger's algorithm can be used, 
 following a different approach compared to Wilf \cite{Wilf1999} and compared to our shifting method, 
 to obtain Chu and Zhang's accelerated series \cite{ChuZhang2014}. 

 Applying Zeilberger's algorithm to \eqref{Fnegativequarter}, we obtain the following rational certificate. 

 \ 

\noindent $R(n, k) = \frac{1}{(a+k+n) (b+k+n)} (a + n) (1 + a + n) (b + n) (1 + b + n) (2 a b - c - a c - b c + 
 c^2 + 2 a k + 2 b k - 2 c k + 2 k^2 + n + 3 a n + 3 b n - 4 c n + 
 6 k n + 5 n^2).$

\begin{theorem}
 For $F(n, k)$ as in \eqref{Fnegativequarter}
 and for the corresponding rational certificate $R(n, k)$
 given above, we write $G(n, k) = R(n, k) F(n, k)$. For 
 $\mathcal{F}(n) = \sum_{k=0}^{\infty} F(n, k)$, 
 the recursion 
\begin{equation}\label{20254u205n20232722272AM22}
 \mathcal{F}(n) = g_{1}(n) + g_{2}(n) \mathcal{F}(n+2) 
\end{equation}
 is satisfied 
 for the rational functions in \eqref{20245050341AMA} 
 and for the following polynomials: 

 \ 

\noindent $p_{1}(n) = a^2 c^2 - 2 a^2 c n + a^2 n^2 - 2 a b c^2 + 4 a b c n - 2 a b n^2 + 
 b^2 c^2 - 2 b^2 c n + b^2 n^2 - c^4 + 4 c^3 n - 6 c^2 n^2 + 4 c n^3 -
 n^4 - a^2 c + a^2 n + 2 a b c - 2 a b n - b^2 c + b^2 n + 3 c^3 - 
 9 c^2 n + 9 c n^2 - 3 n^3 - 3 c^2 + 6 c n - 3 n^2 + c - n,$ 

 \ 

\noindent $p_{2}(n) = -4 a^2 b^2 - 8 a^2 b n - 4 a^2 n^2 - 8 a b^2 n - 16 a b n^2 - 
 8 a n^3 - 4 b^2 n^2 - 8 b n^3 - 4 n^4 - 4 a^2 b - 4 a^2 n - 
 4 a b^2 - 16 a b n - 12 a n^2 - 4 b^2 n - 12 b n^2 - 8 n^3 - 4 a b - 
 4 a n - 4 n b - 4 n^2.$
\end{theorem}

\begin{proof}
 For $p_{1}(n)$ and $p_{2}(n)$ as specified, we may verify that the difference equation 
 $p_{1}(n) F(n+2, k) + p_{2}(n) F(n, k) = G(n, k+1) - G(n, k)$ holds, 
 and we may apply a telescoping argument as before. 
\end{proof}

 The iterative application of the $\mathcal{F}$-recurrence in 
 \eqref{20254u205n20232722272AM22} produces 
\begin{equation}\label{280828480858q0q8qqqqq83qq8q4q8q4q8q2q8qAq8qM8qA}
 \mathcal{F}(n) = \sum_{j=0}^{\infty} \left( \prod_{i=0}^{j-1} g_{2}(n + 2i) \right) g_{1}(n+2j), 
\end{equation}
 by analogy with \eqref{mainacceleration}. 
 For the specified functions $\mathcal{F}$ and $g_{1}$ and $g_{2}$, 
 the formula in \eqref{280828480858q0q8qqqqq83qq8q4q8q4q8q2q8qAq8qM8qA}
 provides a series acceleration, 
 yielding the following formulas for the specified values of the free parameters $a$, $b$, $c$, and $n$. 

 \begin{example}
 Setting $(a, b, c, n) = 
 \big( \frac{2}{3}, 1, -\frac{1}{3}, \frac{2}{3} \big)$, we obtain 
 $$ \frac{64 \log (2)}{3} = 
 \sum_{j=0}^{\infty} \left(-\frac{1}{4} \right)^{j} \left[ \begin{matrix} 
 \frac{1}{2}, 1 \vspace{1mm} \\ 
 \frac{4}{3}, \frac{5}{3} 
 \end{matrix} \right]_{j} \left( 30 j+17 \right). $$ 
\end{example}

\begin{example}
 Setting $(a, b, c, n) = 
 \big( \frac{1}{6}, \frac{1}{3}, \frac{1}{6}, \frac{2}{3} 
 \big)$, we obtain 
 $$ 5 \sqrt{3} = 
 \sum_{j=0}^{\infty} \left(-\frac{1}{4} \right)^{j} \left[ \begin{matrix} 
 \frac{1}{4}, \frac{2}{3}, \frac{3}{4}, \frac{5}{6} \vspace{1mm} \\ 
 \frac{11}{12}, 1, \frac{17}{12}, \frac{3}{2} 
 \end{matrix} \right]_{j} \left( 60 j^2+51 j+10 \right). $$ 
\end{example}

\begin{example}
 Setting $(a, b, c, n) = 
 \big( \frac{1}{6}, \frac{1}{2}, -\frac{5}{6}, \frac{5}{6} \big)$, we obtain 
 $$ \frac{32 \sqrt[3]{2} }{3} = 
 \sum_{j=0}^{\infty} \left(-\frac{1}{4} \right)^{j} \left[ \begin{matrix} 
 \frac{5}{6}, \frac{4}{3} \vspace{1mm} \\ 
 1, \frac{5}{3}
 \end{matrix} \right]_{j} \left( 60 j^2+88 j+33 \right). $$ 
\end{example}

\begin{example}
 Setting $(a, b, c, n) = 
 \big( \frac{1}{6}, \frac{2}{3}, \frac{1}{2}, \frac{1}{3} \big)$, we obtain 
 $$ 9 \sqrt[3]{2} = 
 \sum_{j=0}^{\infty} \left(-\frac{1}{4} \right)^{j} \left[ \begin{matrix} 
 -\frac{1}{12}, \frac{1}{6}, \frac{5}{12}, \frac{2}{3} \vspace{1mm} \\ 
 \frac{3}{4}, 1, \frac{5}{4}, \frac{3}{2} 
 \end{matrix} \right]_{j} \left( 360 j^2+174 j+11 \right). $$ 
\end{example}

\begin{example}
 Setting $(a, b, c, n) = 
 \big( \frac{1}{4}, \frac{1}{2}, \frac{1}{4}, \frac{1}{2} \big)$, we obtain 
 the formula of Ramanujan type highlighted in \eqref{RamanujanType5}. 
\end{example}

\begin{example}
 Setting $(a,b,c,n) = \big( \frac{3}{4}, \frac{1}{2}, \frac{3}{4}, \frac{1}{2} \big)$, we obtain 
 the formula of Ramanujan type highlighted in \eqref{RamanujanType6}. 
\end{example}

\begin{example}
 Setting $(a, b, c, n) = 
 \big( \frac{1}{6}, \frac{1}{3}, \frac{1}{2}, \frac{2}{3} 
 \big)$, we obtain 
 the formula of Ramanujan type highlighted in \eqref{RamanujanType7}. 
\end{example}
 
\begin{example}
 Setting $(a, b, c, n) = 
 \big( \frac{1}{6}, \frac{5}{6}, \frac{1}{6}, \frac{7}{6} \big)$, we obtain 
 the new and motivating $\pi$ formula highlighted in \eqref{PiMotivating3}. 
\end{example}

\subsection{Series of convergence rate $-\frac{1}{27}$}\label{subsectionnegative27}
 What may be seen as a variant of our recursive approach toward the derivation of series of convergence rate $\frac{4}{27}$ in 
 \cite{CampbellLevrieunpublished} has led us to experimentally discover a family of transforms to recover man of Chu and Zhang's series of convergence 
 rate $-\frac{1}{27}$ and to introduce many series of this same convergence rate, via our application of Zeilberger's algorithm to 
\begin{equation}\label{negative27input}
 F(n, k) := \frac{(a)_k (n)_{k+d}}{ (2 n)_{k+c} (2 n)_{k+b}}, 
\end{equation}
 for free parameters $a$, $b$, $c$, and $d$. 
 In this case, an application of Zeilberger's algorithm to \eqref{negative27input}
 produces a corresponding rational certificate reproduced in full in the Appendix below in Section 
 \ref{sectionAppendix}. Despite the unwieldy appearance of 
 this certificate, 
 it can be applied, via our acceleration method, 
 to obtain evaluations for remarkable simple summands and to recover results from Chu and Zhang's work 
 \cite{ChuZhang2014}. 

\begin{theorem}\label{theoremnegative27}
 For $F(n, k)$ as in \eqref{negative27input} and for the corresponding rational certificate $R(n, k)$ given above, we write $G(n, k) = R(n, k) F(n, k)$. 
 For $\mathcal{F}(n) = \sum_{k=0}^{\infty} F(n, k)$, the recursion 
 in \eqref{mathcalFr1} is satisfied for 
 the rational functions in \eqref{20245050341AMA}
 and for the polynomials $p_{1}(n)$ and $p_{2}(n)$
 given in Section \ref{sectionAppendix}. 
\end{theorem}

\begin{proof}
 For $p_{1}(n)$ and $p_{2}(n)$ as specified, we may again verify that the difference equation 
 $ p_{1}(n) F(n+1, k) + p_{2}(n) F(n, k) = G(n, k+1) - G(n, k)$ holds, 
 so as to apply the same telescoping argument as before. 
\end{proof}

 We again obtain the desired vanishing condition in \eqref{desiredvanishing}, 
 so that \eqref{mainacceleration} provides a series acceleration. 
 Specifying rational values for $a$, $b$, $c$, $d$, and $n$ 
 yields the accelerated series below. 

\begin{example}
 Setting $(a, b, c, d, n) = \big( \frac{1}{4}, \frac{1}{2}, \frac{3}{4}, 0, \frac{1}{4} \big)$, we obtain 
 $$ \frac{1155 \pi }{16 \sqrt{2}} = 
 \sum_{j=0}^{\infty} \left(-\frac{1}{27} \right)^{j} \left[ \begin{matrix} 
 \frac{3}{8}, \frac{3}{4}, \frac{7}{8}, 1 \vspace{1mm} \\ 
 \frac{9}{8}, \frac{19}{12}, \frac{13}{8}, \frac{23}{12} 
 \end{matrix} \right]_{j} \left( 448 j^3+976 j^2+697 j+164 \right). $$ 
\end{example}

\begin{example}
 Setting $(a, b, c, d, n) = \big( \frac{1}{2}, \frac{1}{2}, \frac{1}{2}, 0, \frac{1}{2} \big)$, we obtain 
 $$ \frac{27 \pi ^2}{2} = 
 \sum_{j=0}^{\infty} \left(-\frac{1}{27} \right)^{j} \left[ \begin{matrix} 
 \frac{1}{2}, \frac{1}{2}, \frac{1}{2}, 1, 1, 1 \vspace{1mm} \\ 
 \frac{5}{4}, \frac{5}{4}, \frac{4}{3}, \frac{5}{3}, \frac{7}{4}, \frac{7}{4}
 \end{matrix} \right]_{j} \text{{\footnotesize $(896 j^4+2336 j^3+2232 j^2+919 j+136)$}}. $$ 
\end{example}

\begin{example}
 Setting $(a, b, c, d, n) = \big( \frac{1}{2}, 0, -\frac{1}{2}, 0, 1 \big)$, we obtain 
 $$ 24 \log (2) = 
 \sum_{j=0}^{\infty} \left(-\frac{1}{27} \right)^{j} \left[ \begin{matrix} 
 \frac{1}{2}, 1 \vspace{1mm} \\ 
 \frac{4}{3}, \frac{5}{3} 
 \end{matrix} \right]_{j} \left( 28 j+17 \right). $$ 
\end{example}

\begin{example}
 Setting $(a, b, c, d, n) = \big( 1, \frac{1}{2}, 1, 0, \frac{1}{2} \big)$, we obtain 
 $$ 96 \log (2) = 
 \sum_{j=0}^{\infty} \left(-\frac{1}{27} \right)^{j} \left[ \begin{matrix} 
 \frac{1}{2}, \frac{1}{2} \vspace{1mm} \\ 
 \frac{4}{3}, \frac{5}{3}
 \end{matrix} \right]_{j} \frac{896 j^4+2656 j^3+2824 j^2+1268 j+201}{(j+1) (4 j+1) (4 j+3)}. $$ 
\end{example}

\begin{example}
 Setting $(a, b, c, n) = \big( 1, 0, 0, \frac{1}{2}, 1 \big)$, we obtain 
 $$ 30 \log (2) = 
 \sum_{j=0}^{\infty} \left(-\frac{1}{27} \right)^{j} \left[ \begin{matrix} 
 \frac{1}{2}, \frac{1}{2} \vspace{1mm} \\ 
 \frac{7}{6}, \frac{11}{6} 
 \end{matrix} \right]_{j} \frac{28 j^2+48 j+21}{j+1}. $$ 
\end{example}

\begin{example}
 Setting $(a, b, c, d, n) = \big( \frac{1}{4}, \frac{1}{4}, \frac{1}{2}, 0, \frac{1}{4} \big)$, we obtain 
 $$ \frac{15 \sqrt{2}}{2} = 
 \sum_{j=0}^{\infty} \left(-\frac{1}{27} \right)^{j} \left[ \begin{matrix} 
 \frac{1}{4}, \frac{1}{4}, \frac{3}{4} \vspace{1mm} \\ 
 1, \frac{13}{12}, \frac{17}{12} 
 \end{matrix} \right]_{j} \left( 224 j^2+124 j+11 \right). $$ 
\end{example}

\begin{example}
 Setting $(a, b, c, n) = \big( \frac{1}{4}, \frac{1}{2}, 0, \frac{1}{4}, \frac{1}{2} \big)$, we obtain 
 $$ 960 \sqrt{2} = 
 \sum_{j=0}^{\infty} \left(-\frac{1}{27} \right)^{j} \left[ \begin{matrix} 
 \frac{3}{8}, \frac{5}{8}, \frac{3}{4}, \frac{7}{8}, \frac{9}{8} \vspace{1mm} \\ 
 1, \frac{7}{6}, \frac{3}{2}, \frac{3}{2}, \frac{11}{6} 
 \end{matrix} \right]_{j} \text{{\footnotesize $\left( 7168 j^3+12928 j^2+7472 j+1395 \right)$}}. $$ 
\end{example}

\begin{example}
 Setting $(a, b, c, d, n) = \big( \frac{1}{6}, 0, \frac{5}{6}, 0, \frac{1}{2} \big)$, we obtain 
 $$ \frac{1729 \sqrt{3}}{96} = 
 \sum_{j=0}^{\infty} \left(-\frac{1}{27} \right)^{j} \left[ \begin{matrix} 
 \frac{1}{2}, \frac{5}{6}, \frac{4}{3}, \frac{4}{3} \vspace{1mm} \\ 
 1, \frac{25}{18}, \frac{31}{18}, \frac{37}{18} 
 \end{matrix} \right]_{j} \left( 42 j^2+75 j+32 \right). $$ 
\end{example}

\begin{example}
 Setting $(a, b, c, d, n) = \big( \frac{1}{6}, \frac{1}{6}, \frac{1}{3}, 0, \frac{1}{3} \big)$, we obtain 
 $$ 28 \sqrt[3]{2} = 
 \sum_{j=0}^{\infty} \left(-\frac{1}{27} \right)^{j} \left[ \begin{matrix} 
 \frac{1}{3}, \frac{1}{3}, \frac{2}{3}, \frac{5}{6} \vspace{1mm} \\ 
 1, \frac{10}{9}, \frac{13}{9}, \frac{16}{9} 
 \end{matrix} \right]_{j} \left( 756 j^3+1035 j^2+396 j+37 \right). $$ 
\end{example}

\begin{example}
 Setting $(a, b, c, d, n) = \big( \frac{3}{4}, \frac{1}{4}, -\frac{1}{2}, 0, \frac{3}{4} \big)$, we obtain 
 the new and motivating $\pi$ formula highlighted in \eqref{PiMotivating4}. 
\end{example}

 \begin{example}
 Setting $(a, b, c, d, n) = \big( \frac{1}{3}, \frac{2}{3}, 1, \frac{1}{3}, \frac{1}{3} \big)$, we obtain 
 the new and motivating $\pi$ formula highlighted in \eqref{PiMotivating8}. 
\end{example}

 Setting $(a, b, c, d, n)$ to be among the tuples $ (1, 0, 0, 0, 1)$, $ \big( 1, 0, -\frac{1}{2}, 0, 1 \big)$, 
 $ \big( \frac{1}{2}, 0, \frac{1}{2}, 0, \frac{1}{2} \big)$, 
 $ \big( 1, \frac{1}{2}, \frac{1}{2}, 0, \frac{1}{2} \big)$, 
 $ \big( \frac{1}{2}, \frac{1}{2}, \frac{1}{2}, \frac{1}{2}, \frac{1}{2} \big)$, 
 $ \big( \frac{1}{3}, \frac{1}{3}, \frac{2}{3}, 0, \frac{1}{3} \big)$, 
 we recover, in equivalent ways, distinct results from Chu and Zhang's work on accelerating Dougall's sum 
 \cite{ChuZhang2014}. 

 Setting $(a, b, c, d, n)$ to be among the tuples $ \big( \frac{1}{3}, 0, \frac{1}{3}, 0, \frac{2}{3} \big)$ and $\big( \frac{2}{3}, \frac{2}{3}, 1, 0, 
 \frac{1}{3} \big)$, we recover distinct results from Chu's work on a dual version of Dougall's sum \cite{Chu2023}. 

 Setting $(a, b, c, d, n) = \big( \frac{1}{4}, \frac{3}{4}, 0, \frac{3}{4}, \frac{1}{2} \big)$, we obtain the formula of Ramanujan type highlighted in 
 \eqref{onlyequivalent}, which, as indicated above, is equivalent to the Chu series in \eqref{quadraticequiv} proved in 
 both \cite{Chu2023} and \cite{Chu2021Infinite}. 

\subsection{Series of convergence rate $ \frac{4}{27}$}
 For the sake of brevity, the remainder of Section \ref{sectionmain} will be organized in such a way so as to leave it to the reader to verify how the 
 rational certificates and polynomial recurrences associated with the specified $F(n, k)$-functions may be applied as in our applications of Theorems 
 \ref{theoremquarter}--\ref{theoremnegative27}. In this direction, as to apply our shifted method to extend how the first-order recurrence for 
 \eqref{previousfor4over27} was applied in our past work \cite{CampbellLevrieunpublished} to produce accelerated series of convergence rate 
 $\frac{4}{27}$, this leads us to define the shifted family 
\begin{equation}\label{shifted4over27}
 F(n, k) := \frac{(a)_k (b)_k}{(n)_k (2n)_{c+k}}. 
\end{equation}
   Applying Zeilberger's algorithm to \eqref{shifted4over27}, we obtain the $r = 1$ case of \eqref{firstorder}, with \eqref{desiredvanishing} again vanishing,   
     so that \eqref{mainacceleration} provides a series acceleration.     By inputting the below specified combinations of rational values        for the parameters    
   $a$, $b$, $c$, and $n$    into \eqref{mainacceleration},     we obtain the corresponding closed forms indicated below.    

\begin{example}
 Setting $(a, b, c, n) =  \big( \frac{1}{12}, \frac{5}{12}, \frac{1}{4}, 1 \big)$, we obtain  $$ \frac{885735 \sqrt{2} }{64} =  \sum_{j=0}^{\infty} 
 \left(\frac{4}{27} \right)^{j} \left[ \begin{matrix}  \frac{7}{12}, \frac{11}{12}, \frac{13}{12}, \frac{17}{12} \vspace{1mm} \\  1, \frac{13}{8}, 
 \frac{17}{8}, \frac{9}{4}  \end{matrix} \right]_{j} \text{{\footnotesize $\left( 19872 j^3+58788 j^2+55923 j+16796 \right)$}}. $$ 
\end{example} 

\begin{example}
 Setting $(a, b, c, n) =  \big( \frac{1}{4}, \frac{1}{4}, -\frac{1}{2}, \frac{3}{4} \big)$, we obtain  $$ 45 \sqrt{2} =  \sum_{j=0}^{\infty} 
 \left(\frac{4}{27} \right)^{j} \left[ \begin{matrix} 
 \frac{3}{8}, \frac{3}{8}, \frac{1}{2}, \frac{7}{8}, \frac{7}{8} \vspace{1mm} \\ 
 \frac{3}{4}, 1, \frac{13}{12}, \frac{17}{12}, \frac{7}{4} 
 \end{matrix} \right]_{j} \left( 1472 j^3+1840 j^2+618 j+45 \right). $$ 
\end{example}

\begin{example}
  Setting $(a, b, c, n) =   \big( \frac{1}{4}, \frac{1}{4}, -\frac{1}{4}, 1 \big)$, we obtain   $$ \frac{4095 \sqrt{2} }{16} =   \sum_{j=0}^{\infty}  
  \left(\frac{4}{27} \right)^{j} \left[ \begin{matrix}   \frac{3}{4}, \frac{3}{4}, \frac{3}{4}, \frac{5}{4}, \frac{5}{4} \vspace{1mm} \\  1, \frac{11}{8},  
  \frac{17}{12}, \frac{15}{8}, \frac{25}{12}   \end{matrix} \right]_{j} \left( 736 j^3+1748 j^2+1309 j+300 \right). $$  
\end{example}

\begin{example}
  Setting $(a, b, c, n) =   \big( \frac{1}{12}, \frac{5}{12}, -\frac{3}{4}, 1 \big)$, we obtain    $$ \frac{3645 \sqrt{2}}{16} =   \sum_{j=0}^{\infty}  
  \left(\frac{4}{27} \right)^{j} \left[ \begin{matrix}   \frac{5}{12}, \frac{7}{12}, \frac{11}{12}, \frac{13}{12} \vspace{1mm} \\   1, \frac{9}{8},  
  \frac{5}{4}, \frac{13}{8}   \end{matrix} \right]_{j} \left( 1656 j^2+1449 j+260 \right). $$  
\end{example}

\begin{example}
  Setting $(a, b, c, n) =   \big( \frac{1}{4}, \frac{3}{4}, -\frac{1}{2}, 1 \big)$, we obtain   $$ 240 \sqrt{2} =   \sum_{j=0}^{\infty} \left(\frac{4}{27}  
  \right)^{j} \left[ \begin{matrix}   \frac{3}{8}, \frac{5}{8}, \frac{7}{8}, \frac{9}{8} \vspace{1mm} \\   1, \frac{7}{6}, \frac{3}{2}, \frac{11}{6}  
  \end{matrix} \right]_{j} \left( 736 j^2+984 j+315 \right). $$  
\end{example}

\begin{example}
 Setting $(a, b, c, n) = \big( \frac{1}{6}, \frac{5}{6}, \frac{1}{2}, 1 \big)$, we obtain $$ \frac{98415 \sqrt{3}}{256} = \sum_{j=0}^{\infty} 
 \left(\frac{4}{27} \right)^{j} \left[ \begin{matrix} \frac{1}{6}, \frac{5}{6}, \frac{4}{3}, \frac{5}{3} \vspace{1mm} \\ 1, \frac{3}{2}, \frac{7}{4}, 
 \frac{9}{4} \end{matrix} \right]_{j} \left( 828 j^2+1521 j+640 \right). $$ 
\end{example}

\begin{example}
 Setting $(a, b, c, n) = \big( \frac{1}{6}, \frac{1}{6}, -\frac{1}{3}, \frac{2}{3} \big)$, we obtain $$ \frac{896 \sqrt{3} }{9} = \sum_{j=0}^{\infty} 
 \left(\frac{4}{27} \right)^{j} \left[ \begin{matrix} \frac{5}{12}, \frac{5}{12}, \frac{1}{2}, \frac{11}{12}, \frac{11}{12} \vspace{1mm} \\ \frac{2}{3}, 
 1, \frac{10}{9}, \frac{13}{9}, \frac{16}{9} \end{matrix} \right]_{j} \left( 3312 j^3+4320 j^2+1510 j+109 \right). $$ 
\end{example}

\begin{example}
 Setting $(a, b, c, n) = \big( \frac{1}{6}, \frac{1}{6}, -\frac{1}{3}, 1 \big)$, we obtain $$ \frac{71680 \sqrt{3} }{81} = \sum_{j=0}^{\infty} 
 \left(\frac{4}{27} \right)^{j} \left[ \begin{matrix} \frac{3}{4}, \frac{3}{4}, \frac{5}{6}, \frac{5}{4}, \frac{5}{4} \vspace{1mm} \\ 1, \frac{4}{3}, 
 \frac{13}{9}, \frac{16}{9}, \frac{19}{9} \end{matrix} \right]_{j} \text{{\footnotesize $\left( 3312 j^3+7632 j^2+5494 j+1215 \right)$}}. $$ 
\end{example}

\begin{example}
 Setting $(a, b, c, n) = \big( \frac{1}{6}, \frac{1}{3}, -\frac{2}{3}, \frac{5}{6} \big)$, we obtain $$ 28 \sqrt[3]{2} = \sum_{j=0}^{\infty} 
 \left(\frac{4}{27} \right)^{j} \left[ \begin{matrix} \frac{1}{3}, \frac{5}{12}, \frac{2}{3}, \frac{11}{12} \vspace{1mm} \\ 1, \frac{10}{9}, 
 \frac{13}{9}, \frac{16}{9} \end{matrix} \right]_{j} \left( 621 j^3+828 j^2+303 j+26 \right). $$ 
\end{example}

\begin{example}
 Setting $(a, b, c, n) = 
 \big( \frac{2}{3}, \frac{1}{6}, -\frac{2}{3}, 1 \big)$, we obtain 
 $$ 81 \sqrt[3]{2} = 
 \sum_{j=0}^{\infty} \left(\frac{4}{27} \right)^{j} \left[ \begin{matrix} 
 \frac{1}{3}, \frac{1}{3}, \frac{7}{12}, \frac{5}{6}, \frac{13}{12} \vspace{1mm} \\ 
 1, \frac{7}{6}, \frac{7}{6}, \frac{3}{2}, \frac{5}{3}
 \end{matrix} \right]_{j} \left( 1242 j^3+1728 j^2+738 j+91 \right). $$ 
\end{example}

\begin{example}
 Setting $(a, b, c, n) = \big( \frac{1}{2}, \frac{1}{2}, -\frac{1}{2}, 1 \big)$, we obtain the new and motivating $\pi$ formula 
 highlighted in \eqref{PiMotivating6}. 
\end{example}

\subsection{Series of convergence rate $\frac{16}{27}$}
 Setting 
\begin{equation}\label{Ffor1627}
 F(n, k) := \binom{n}{k} \frac{ (1)_{k} }{ (3n+b)_{k + a}} 
\end{equation}
 and applying Zeilberger's algorithm, this again produces a first-order recurrence, for the $r = 1$ case of \eqref{firstorder}, with \eqref{desiredvanishing} 
 again vanishing. So, by plugging \eqref{Ffor1627} and its associated $g_{1}$- and $g_{2}$-functions into \eqref{mainacceleration}, the 
 hypergeometric transform we obtain provides an acceleration of the rate of convergence. Specifying rational values for the free parameters $a$, $b$, 
 and $n$, as below, produces the corresponding closed forms shown below. 

\begin{example}
 Setting $(a, b, n) = \big( -1, \frac{1}{2}, \frac{1}{2} \big)$, we obtain $$ 384 \sqrt{2} = \sum_{j=0}^{\infty} \left( \frac{16}{27} \right)^{j} \left[ 
 \begin{matrix} \frac{3}{8}, \frac{5}{8}, \frac{7}{8}, \frac{9}{8} \vspace{1mm} \\ 1, \frac{4}{3}, \frac{5}{3}, \frac{5}{2} \end{matrix} \right]_{j} 
 \left( 704 j^3+2096 j^2+1564 j+321 \right). $$ 
\end{example}

\begin{example}
 Setting $(a, b, n) = \big( -\frac{1}{3}, \frac{1}{3}, \frac{1}{3} \big)$, we obtain $$ 144 \sqrt[3]{2} = \sum_{j=0}^{\infty} \left( \frac{16}{27} 
 \right)^{j} \left[ \begin{matrix} \frac{7}{12}, \frac{5}{6}, \frac{13}{12} \vspace{1mm} \\ 1, \frac{7}{3}, \frac{5}{3} \end{matrix} \right]_{j} \left( 
 99 j^2+237 j+113 \right). $$ 
\end{example}
 
 We may also recover results from Chu and Zhang's paper \cite{ChuZhang2014}, as in the $(a, b, n) = \big( \frac{1}{2}, \frac{1}{2}, \frac{1}{2} \big)$ 
 case, which produces Example 47 from \cite{ChuZhang2014}. 

 By mimicking the derivation of our hypergeometric transform obtained from \eqref{Ffor1627}, but with $ \binom{n}{k} \frac{ (2)_{k} }{ (3n+b,a+k) } $ 
 used in place of \eqref{Ffor1627}, and by specifying $a$, $b$, and $n$ as below, this results in the correpsonding closed forms 
 in Examples \ref{192sqrt2}--\ref{14cuberoot}. 

\begin{example}\label{192sqrt2}
 Setting $(a, b, n) = \big( \frac{1}{4}, \frac{1}{4}, \frac{1}{2} \big)$, we obtain $$ 192 \sqrt{2} = \sum_{j=0}^{\infty} \left( \frac{16}{27} \right)^{j} 
 \left[ \begin{matrix} -\frac{3}{4}, \frac{5}{8}, \frac{7}{8}, \frac{9}{8}, \frac{11}{8} \vspace{1mm} \\ 1, \frac{5}{4}, \frac{4}{3}, \frac{5}{3}, 
 \frac{5}{2} \end{matrix} \right]_{j} \left( 704 j^3+2416 j^2+2252 j+633 \right). $$ 
\end{example}

\begin{example}\label{14cuberoot}
 Setting $(a, b, n) = \big( -\frac{1}{3}, \frac{1}{3}, \frac{1}{3} \big)$, we obtain $$ -14 \sqrt[3]{2} = \sum_{j=0}^{\infty} \left( \frac{16}{27} 
 \right)^{j} \left[ \begin{matrix} -\frac{7}{6}, \frac{7}{12}, \frac{13}{12} \vspace{1mm} \\ \frac{1}{3}, \frac{2}{3}, 1 \end{matrix} \right]_{j} 
 \frac{66 j+133}{3 j+4}. $$ 
\end{example}

\subsection{Series of convergence rate $\frac{1}{64}$}
 Setting $ F(n, k) := \frac{(a)_k (2 n)_k}{(3 n)_{b+k} (3 n)_{c+k}} $ and mimicking the acceleration method applied in the preceding subsections, we 
 may obtain the following results, by specifying the parameters $a$, $b$, $c$, and $n$ as indicated. 

\begin{example}
 Setting $(a, b, c, n) = \big( \frac{1}{2}, \frac{1}{4}, \frac{3}{4}, \frac{1}{4} \big)$, we obtain $$ \frac{105 \pi }{4} = 
 \sum_{j=0}^{\infty} \left(\frac{1}{64} \right)^{j} \left[ \begin{matrix} 
 \frac{1}{4}, \frac{1}{2}, \frac{3}{4}, 1 \vspace{1mm} \\ 
 \frac{9}{8}, \frac{11}{8}, \frac{13}{8}, \frac{15}{8} 
 \end{matrix} \right]_{j} (3 j+2) \left(112 j^2+144 j+41\right). $$ 
\end{example}

\begin{example}
 Setting $(a, b, c, n) = \big( \frac{1}{2}, \frac{3}{4}, \frac{3}{4}, \frac{1}{4} \big)$, we obtain 
 $$ \frac{75 \pi ^2}{16} = 
 \sum_{j=0}^{\infty} \left(\frac{1}{64} \right)^{j} \left[ \begin{matrix} 
 \frac{1}{3}, \frac{1}{3}, 1, 1, 1, \frac{5}{3}, \frac{5}{3} \vspace{1mm} \\ 
 \frac{7}{6}, \frac{7}{6}, \frac{3}{2}, \frac{3}{2}, \frac{3}{2}, \frac{11}{6}, \frac{11}{6} 
 \end{matrix} \right]_{j} (3 j+2) \left(63 j^2+81 j+23\right). $$ 
\end{example}

\begin{example}
 Setting $(a, b, c, n) = \big( 1, \frac{1}{2}, \frac{1}{2}, \frac{1}{2} \big)$, we obtain the new and motivating result highlighted in \eqref{PiMotivating1}. 
\end{example}

 Similarly, by setting $ F(n, k) := \frac{ (n)_{k} (n)_{k+c} }{ (2n)_{k + b} (2n)_{k+a} } $ and again mimicking the acceleration methods given above, 
 we may obtain the following, according to the specified parameter values. 

\begin{example}
 Setting $(a, b, c, n) = \big( -\frac{5}{6}, 0, -\frac{2}{3}, 1 \big)$, we obtain $$ \frac{385 \pi }{12 \sqrt{3}} = 
 \sum_{j=0}^{\infty} \left(\frac{1}{64} \right)^{j} \left[ \begin{matrix} 
 \frac{1}{6}, \frac{1}{3}, \frac{5}{6}, 1, \frac{5}{3} \vspace{1mm} \\ 
 \frac{13}{12}, \frac{17}{12}, \frac{3}{2}, \frac{19}{12}, \frac{23}{12} 
 \end{matrix} \right]_{j} (3 j+2) \left(84 j^2+108 j+29\right). $$ 
\end{example}

\begin{example}
 Setting $(a, b, c, n) = \big( -\frac{3}{4}, \frac{1}{2}, -1, \frac{1}{4} \big)$, we obtain 
 $$ -10 \sqrt{2} = 
 \sum_{j=0}^{\infty} \left(\frac{1}{64} \right)^{j} \left[ \begin{matrix} 
 -\frac{3}{4}, -\frac{1}{2}, \frac{1}{4}, \frac{3}{4}, \frac{7}{4} \vspace{1mm} \\ 
 \frac{3}{8}, \frac{7}{8}, 1, \frac{9}{8}, \frac{13}{8} 
 \end{matrix} \right]_{j} \left( 672 j^3+304 j^2-70 j-17 \right). $$ 
\end{example}
 
\begin{example}
 Setting $(a, b, c, n) = \big( -\frac{2}{3}, -\frac{1}{6}, -\frac{1}{2}, \frac{5}{6} \big)$, we obtain 
 $$ 162 \sqrt[3]{2} = 
 \sum_{j=0}^{\infty} \left(\frac{1}{64} \right)^{j} \left[ \begin{matrix} 
 \frac{1}{6}, \frac{1}{3}, \frac{2}{3}, \frac{5}{6} \vspace{1mm} \\ 
 1, \frac{5}{4}, \frac{3}{2}, \frac{7}{4} 
 \end{matrix} \right]_{j} \left( 2268 j^3+3420 j^2+1557 j+203 \right). $$ 
\end{example}
 
\begin{example}
 Setting $(a, b, c, n) = \big( -\frac{2}{3}, -\frac{1}{3}, -1, \frac{5}{6} \big)$, we obtain the formula of Ramanujan type 
 highlighted in \eqref{RamanujanType8}. 
\end{example}

\begin{example}
 Setting $(a, b, c, n) = \big( \frac{1}{2}, \frac{1}{2}, 0, \frac{1}{2} \big)$, we obtain 
 the new and motivating $\pi$ formula highlighted in \eqref{PiMotivating9}. 
\end{example}

 Setting $(a, b, c, n)$ as the tuples $\big( 0, 0, -1, \frac{1}{2} \big)$, $(0, 0, 0, 1)$, and $\big( -\frac{2}{3}, 0, -\frac{1}{3}, 1 \big)$, we recover 
 results from Chu and Zhang's accelerations of Dougall's sum, including Ramanujan's series of convegence rate $\frac{1}{64}$, 
 and including a series previously given by Guillera \cite{Guillera2008}. 

 We may similarly recover many of the results from Chu's work on Gould--Hsu inverse series relations \cite{Chu2021Ramanujan}, 
 by setting $(a, b, c, n)$ as $\big( \frac{1}{2}, 0, 0, \frac{1}{2} \big)$, 
 $\big( -\frac{2}{3}, \frac{1}{6}, -1, \frac{5}{6} \big)$, 
 $\big( -\frac{2}{3}, \frac{1}{3}, -\frac{2}{3}, \frac{5}{6} \big)$, 
 $\big( \frac{1}{3}, \frac{1}{3}, -\frac{2}{3}, \frac{1}{3} \big)$, 
 $\big( -\frac{1}{2}, \frac{1}{4}, -1, \frac{3}{4} \big)$, and 
 $\big( -\frac{1}{2}, \frac{1}{2}, -\frac{1}{2}, \frac{3}{4} \big)$, 
 with these different parameter combinations yielding distinct results from \cite{Chu2021Ramanujan}. 
 
\subsection{Series of convergence rate $\frac{27}{64}$}
 Finally, we set $$ \frac{(n)_k (2 n)_k}{(4 n)_{b+k} (a+n)_{c+k}} $$ to obtain series of convergence rat e $\frac{27}{64}$, 
 by mimicking the acceleration methods given above, 
 and by setting the parameters $a$, $b$, $c$, and $n$ as below. 

\begin{example}
 Setting $(a, b, c, n) = \big( \frac{3}{8}, 0, \frac{3}{8}, \frac{1}{2} \big)$, we obtain 
 $$ \frac{189 \pi }{4} = 
 \sum_{j=0}^{\infty} \left(\frac{27}{64} \right)^{j} \left[ \begin{matrix} 
 \frac{1}{2}, \frac{1}{2}, \frac{1}{2}, \frac{5}{6}, 1, \frac{7}{6} \vspace{1mm} \\ 
 \frac{5}{4}, \frac{11}{8}, \frac{7}{4}, \frac{7}{4}, \frac{15}{8} 
 \end{matrix} \right]_{j} \text{{\footnotesize $ \left( 592 j^4+1656 j^3+1700 j^2+753 j+120 \right)$}}. $$
\end{example}

\begin{example}
 Setting $(a, b, c, n) = \big( 1, 0, 0, \frac{1}{2} \big)$, we obtain $$ 192 \log (2) = 
 \sum_{j=0}^{\infty} \left(\frac{27}{64} \right)^{j} \left[ \begin{matrix} 
 \frac{5}{6}, \frac{7}{6} \vspace{1mm} \\ 
 \frac{5}{4}, \frac{7}{4} 
 \end{matrix} \right]_{j} \frac{148 j^2+256 j+111}{(j+1) (2 j+1)}. $$ 
\end{example}
 
\begin{example}
 Setting $(a, b, c, n) = \big( 0, 0, 1, 1 \big)$, we obtain $$ 30 = 
 \sum_{j=0}^{\infty} \left(\frac{27}{64} \right)^{j} \left[ \begin{matrix} 
 1, \frac{4}{3}, \frac{5}{3} \vspace{1mm} \\ 
 \frac{7}{4}, 2, \frac{9}{4} 
 \end{matrix} \right]_{j} \frac{37 j^2+101 j+69}{2 j+3} $$ 
\end{example}

 Setting $(a, b, c, n) = \big( \frac{1}{4}, 0, \frac{1}{4}, 1 \big)$
 also allows us to recover 
 Example 51 from Chu and Zhang's accelerations from Dougall's sum \cite{ChuZhang2014}. 

\section{Further rates of acceleration}
 The convergence rates involved above agree with the convergence rates in Chu and Zhang's paper \cite{ChuZhang2014}. 
 We briefly conclude by emphasizing the versatility of our method, 
 by noting how it can be further applied to 
 obtain natural rates of convergence not considered in \cite{ChuZhang2014}. 
 For example, setting 
 $$ F(n, k) := \frac{(a)_k (n)_k}{(3 n)_{b+k} (2 n)_{c+k}}, $$
 we can obtain 
 $$ \frac{8800 \pi }{21 \sqrt{3}} = 
 \sum_{j=0}^{\infty} \left( -\frac{1}{64} \right)^{j} \left[ \begin{matrix} 
 1, \frac{10}{9}, \frac{13}{9}, \frac{3}{2}, \frac{5}{3}, \frac{16}{9} \vspace{1mm} \\ 
 \frac{7}{6}, \frac{14}{9}, \frac{7}{4}, \frac{17}{9}, \frac{20}{9}, \frac{9}{4} 
 \end{matrix} \right]_{j} \left( 585 j^3+1956 j^2+2155 j+780 \right). $$ 
 by setting $(a, b, c, n) = \big( -\frac{2}{3}, \frac{2}{3}, 0, \frac{2}{3} \big)$. 
 Similarly, setting 
 $$ F(n, k) := \binom{n}{k} \frac{ (1)_{k} }{ (2 k + b)_{k + a}}, $$ we can obtain 
 $$ \frac{45 \pi }{4} = 
 \sum_{j=0}^{\infty} \left(\frac{27}{32} \right)^{j} \left[ \begin{matrix} 
 1, \frac{4}{3}, \frac{5}{3} \vspace{1mm} \\ 
 \frac{7}{4}, \frac{9}{4}, \frac{5}{2}
 \end{matrix} \right]_{j} \left( 5 j^2+16 j+12 \right) $$ 
 by setting $(a, b, c, n) = \big( 1, \frac{1}{2}, \frac{1}{2} \big)$. 

\subsection*{Acknowledgements}
 The author was supported through a Killam Postdoctoral Fellowship from the Killam Trusts. 
 The author wants to thank Paul Levrie and Karl Dilcher for useful and engaging discussions related to this paper.

\subsection*{Statements and Declarations}
  There are no    competing interests to   disclose.

\section{Appendix}\label{sectionAppendix}
 The rational certificate required in Section \ref{subsectionnegative27} 
 is as below. 

 \ 

\noindent $R(n, k) = \frac{1}{(b+k+2 n) (c+k+2 n)} 4 n (1 + 2 n)^2 (a b c - a^3 b c - b^2 c + 3 a^2 b^2 c - 3 a b^3 c + 
 b^4 c - b c^2 + 3 a^2 b c^2 - 6 a b^2 c^2 + 3 b^3 c^2 - 
 3 a b c^3 + 3 b^2 c^3 + b c^4 - a^2 b d + a^3 b d + a b^2 d - 
 2 a^2 b^2 d + a b^3 d - a^2 c d + a^3 c d + b c d + a b c d - 
 5 a^2 b c d + 7 a b^2 c d - 3 b^3 c d + a c^2 d - 2 a^2 c^2 d + 
 7 a b c^2 d - 5 b^2 c^2 d + a c^3 d - 3 b c^3 d + a^2 d^2 - 
 a^3 d^2 - a b d^2 + 2 a^2 b d^2 - a b^2 d^2 - a c d^2 + 
 2 a^2 c d^2 - 4 a b c d^2 + 2 b^2 c d^2 - a c^2 d^2 + 2 b c^2 d^2 +
 a b k - a^2 b k - b^2 k + a b^2 k + a^2 b^2 k - 2 a b^3 k + 
 b^4 k + a c k - a^2 c k - 2 b c k + a b c k + 4 a^2 b c k - 
 8 a b^2 c k + 4 b^3 c k - c^2 k + a c^2 k + a^2 c^2 k - 
 8 a b c^2 k + 7 b^2 c^2 k - 2 a c^3 k + 4 b c^3 k + c^4 k + b d k -
 a b d k - 2 a^2 b d k + b^2 d k + 4 a b^2 d k - 2 b^3 d k + 
 c d k - a c d k - 2 a^2 c d k + b c d k + 10 a b c d k - 
 8 b^2 c d k + c^2 d k + 4 a c^2 d k - 8 b c^2 d k - 2 c^3 d k + 
 a^2 d^2 k - b d^2 k - 2 a b d^2 k + b^2 d^2 k - c d^2 k - 
 2 a c d^2 k + 4 b c d^2 k + c^2 d^2 k + a k^2 - a^2 k^2 - b k^2 + 
 2 a^2 b k^2 + b^2 k^2 - 4 a b^2 k^2 + 2 b^3 k^2 - c k^2 + 
 2 a^2 c k^2 + b c k^2 - 7 a b c k^2 + 5 b^2 c k^2 + c^2 k^2 - 
 4 a c^2 k^2 + 5 b c^2 k^2 + 2 c^3 k^2 + d k^2 - a d k^2 - 
 a^2 d k^2 + 5 a b d k^2 - 4 b^2 d k^2 + 5 a c d k^2 - 
 7 b c d k^2 - 4 c^2 d k^2 - d^2 k^2 - a d^2 k^2 + 2 b d^2 k^2 + 
 2 c d^2 k^2 - a k^3 + a^2 k^3 + b k^3 - 2 a b k^3 + b^2 k^3 + 
 c k^3 - 2 a c k^3 + 2 b c k^3 + c^2 k^3 - d k^3 + 2 a d k^3 - 
 2 b d k^3 - 2 c d k^3 + d^2 k^3 + 2 a b n - a^2 b n - a^3 b n - 
 2 b^2 n + a b^2 n + 4 a^2 b^2 n - 5 a b^3 n + 2 b^4 n + 2 a c n - 
 a^2 c n - a^3 c n - 7 b c n + a b c n + 19 a^2 b c n - 
 35 a b^2 c n + 17 b^3 c n - 2 c^2 n + a c^2 n + 4 a^2 c^2 n - 
 35 a b c^2 n + 31 b^2 c^2 n - 5 a c^3 n + 17 b c^3 n + 2 c^4 n - 
 2 a^2 d n + 2 a^3 d n + 2 b d n + 4 a b d n - 14 a^2 b d n + 
 18 a b^2 d n - 6 b^3 d n + 2 c d n + 4 a c d n - 14 a^2 c d n + 
 48 a b c d n - 34 b^2 c d n + 18 a c^2 d n - 34 b c^2 d n - 
 6 c^3 d n - 4 a d^2 n + 8 a^2 d^2 n - 12 a b d^2 n + 4 b^2 d^2 n - 
 12 a c d^2 n + 16 b c d^2 n + 4 c^2 d^2 n + 4 a k n - 4 a^2 k n - 
 7 b k n + 5 a b k n + 10 a^2 b k n + b^2 k n - 24 a b^2 k n + 
 14 b^3 k n - 7 c k n + 5 a c k n + 10 a^2 c k n + b c k n - 
 54 a b c k n + 44 b^2 c k n + c^2 k n - 24 a c^2 k n + 
 44 b c^2 k n + 14 c^3 k n + 4 d k n - 4 a d k n - 6 a^2 d k n + 
 4 b d k n + 32 a b d k n - 26 b^2 d k n + 4 c d k n + 
 32 a c d k n - 56 b c d k n - 26 c^2 d k n - 4 d^2 k n - 
 8 a d^2 k n + 12 b d^2 k n + 12 c d^2 k n - 3 k^2 n - a k^2 n + 
 7 a^2 k^2 n + 6 b k^2 n - 25 a b k^2 n + 18 b^2 k^2 n + 
 6 c k^2 n - 25 a c k^2 n + 33 b c k^2 n + 18 c^2 k^2 n - 
 2 d k^2 n + 18 a d k^2 n - 26 b d k^2 n - 26 c d k^2 n + 
 8 d^2 k^2 n + 3 k^3 n - 6 a k^3 n + 6 b k^3 n + 6 c k^3 n - 
 6 d k^3 n + 4 a n^2 - 3 a^2 n^2 - a^3 n^2 - 10 b n^2 + 5 a b n^2 + 
 20 a^2 b n^2 - 41 a b^2 n^2 + 22 b^3 n^2 - 10 c n^2 + 5 a c n^2 + 
 20 a^2 c n^2 - 116 a b c n^2 + 96 b^2 c n^2 - 41 a c^2 n^2 + 
 96 b c^2 n^2 + 22 c^3 n^2 + 4 d n^2 + 4 a d n^2 - 20 a^2 d n^2 + 
 72 a b d n^2 - 48 b^2 d n^2 + 72 a c d n^2 - 120 b c d n^2 - 
 48 c^2 d n^2 - 24 a d^2 n^2 + 24 b d^2 n^2 + 24 c d^2 n^2 - 
 12 k n^2 + 8 a k n^2 + 17 a^2 k n^2 + 5 b k n^2 - 86 a b k n^2 + 
 73 b^2 k n^2 + 5 c k n^2 - 86 a c k n^2 + 148 b c k n^2 + 
 73 c^2 k n^2 + 4 d k n^2 + 56 a d k n^2 - 96 b d k n^2 - 
 96 c d k n^2 + 24 d^2 k n^2 + 11 k^2 n^2 - 41 a k^2 n^2 + 
 56 b k^2 n^2 + 56 c k^2 n^2 - 44 d k^2 n^2 + 9 k^3 n^2 - 12 n^3 + 
 8 a n^3 + 20 a^2 n^3 - 108 a b n^3 + 92 b^2 n^3 - 108 a c n^3 + 
 216 b c n^3 + 92 c^2 n^3 + 80 a d n^3 - 128 b d n^3 - 
 128 c d n^3 + 32 d^2 n^3 + 8 k n^3 - 96 a k n^3 + 164 b k n^3 + 
 164 c k n^3 - 112 d k n^3 + 60 k^2 n^3 - 88 a n^4 + 168 b n^4 + 
 168 c n^4 - 112 d n^4 + 136 k n^4 + 112 n^5).$

 \ 

\noindent The required polynomials for Theorem \ref{theoremnegative27} are as below. 

 \ 

\noindent $p_{1}(n) = -a^2 b c + 2 a^3 b c - a^4 b c + a b^2 c - 3 a^2 b^2 c + 2 a^3 b^2 c +
 a b^3 c - a^2 b^3 c + a b c^2 - 3 a^2 b c^2 + 2 a^3 b c^2 - 
 b^2 c^2 + 4 a b^2 c^2 - 4 a^2 b^2 c^2 - b^3 c^2 + 2 a b^3 c^2 + 
 a b c^3 - a^2 b c^3 - b^2 c^3 + 2 a b^2 c^3 - b^3 c^3 + a^2 b d - 
 2 a^3 b d + a^4 b d - a b^2 d + 3 a^2 b^2 d - 2 a^3 b^2 d - a b^3 d +
 a^2 b^3 d + a^2 c d - 2 a^3 c d + a^4 c d - 2 a b c d + 
 6 a^2 b c d - 4 a^3 b c d + b^2 c d - 5 a b^2 c d + 5 a^2 b^2 c d + 
 b^3 c d - 2 a b^3 c d - a c^2 d + 3 a^2 c^2 d - 2 a^3 c^2 d + 
 b c^2 d - 5 a b c^2 d + 5 a^2 b c^2 d + 2 b^2 c^2 d - 4 a b^2 c^2 d +
 b^3 c^2 d - a c^3 d + a^2 c^3 d + b c^3 d - 2 a b c^3 d + 
 b^2 c^3 d - a^2 d^2 + 2 a^3 d^2 - a^4 d^2 + a b d^2 - 3 a^2 b d^2 + 
 2 a^3 b d^2 + a b^2 d^2 - a^2 b^2 d^2 + a c d^2 - 3 a^2 c d^2 + 
 2 a^3 c d^2 - b c d^2 + 4 a b c d^2 - 4 a^2 b c d^2 - b^2 c d^2 + 
 2 a b^2 c d^2 + a c^2 d^2 - a^2 c^2 d^2 - b c^2 d^2 + 2 a b c^2 d^2 -
 b^2 c^2 d^2 - a^2 b n + 2 a^3 b n - a^4 b n + a b^2 n - 
 3 a^2 b^2 n + 2 a^3 b^2 n + a b^3 n - a^2 b^3 n - a^2 c n + 
 2 a^3 c n - a^4 c n + 6 a b c n - 18 a^2 b c n + 12 a^3 b c n - 
 3 b^2 c n + 17 a b^2 c n - 17 a^2 b^2 c n - 3 b^3 c n + 6 a b^3 c n +
 a c^2 n - 3 a^2 c^2 n + 2 a^3 c^2 n - 3 b c^2 n + 17 a b c^2 n - 
 17 a^2 b c^2 n - 10 b^2 c^2 n + 20 a b^2 c^2 n - 5 b^3 c^2 n + 
 a c^3 n - a^2 c^3 n - 3 b c^3 n + 6 a b c^3 n - 5 b^2 c^3 n + 
 2 a^2 d n - 4 a^3 d n + 2 a^4 d n - 6 a b d n + 18 a^2 b d n - 
 12 a^3 b d n + 2 b^2 d n - 14 a b^2 d n + 14 a^2 b^2 d n + 
 2 b^3 d n - 4 a b^3 d n - 6 a c d n + 18 a^2 c d n - 12 a^3 c d n + 
 6 b c d n - 32 a b c d n + 32 a^2 b c d n + 12 b^2 c d n - 
 24 a b^2 c d n + 4 b^3 c d n + 2 c^2 d n - 14 a c^2 d n + 
 14 a^2 c^2 d n + 12 b c^2 d n - 24 a b c^2 d n + 10 b^2 c^2 d n + 
 2 c^3 d n - 4 a c^3 d n + 4 b c^3 d n + 4 a d^2 n - 12 a^2 d^2 n + 
 8 a^3 d^2 n - 2 b d^2 n + 12 a b d^2 n - 12 a^2 b d^2 n - 
 2 b^2 d^2 n + 4 a b^2 d^2 n - 2 c d^2 n + 12 a c d^2 n - 
 12 a^2 c d^2 n - 8 b c d^2 n + 16 a b c d^2 n - 4 b^2 c d^2 n - 
 2 c^2 d^2 n + 4 a c^2 d^2 n - 4 b c^2 d^2 n - a^2 n^2 + 2 a^3 n^2 - 
 a^4 n^2 + 5 a b n^2 - 15 a^2 b n^2 + 10 a^3 b n^2 - 2 b^2 n^2 + 
 13 a b^2 n^2 - 13 a^2 b^2 n^2 - 2 b^3 n^2 + 4 a b^3 n^2 + 
 5 a c n^2 - 15 a^2 c n^2 + 10 a^3 c n^2 - 9 b c n^2 + 52 a b c n^2 - 
 52 a^2 b c n^2 - 23 b^2 c n^2 + 46 a b^2 c n^2 - 8 b^3 c n^2 - 
 2 c^2 n^2 + 13 a c^2 n^2 - 13 a^2 c^2 n^2 - 23 b c^2 n^2 + 
 46 a b c^2 n^2 - 25 b^2 c^2 n^2 - 2 c^3 n^2 + 4 a c^3 n^2 - 
 8 b c^3 n^2 - 8 a d n^2 + 24 a^2 d n^2 - 16 a^3 d n^2 + 8 b d n^2 - 
 48 a b d n^2 + 48 a^2 b d n^2 + 16 b^2 d n^2 - 32 a b^2 d n^2 + 
 4 b^3 d n^2 + 8 c d n^2 - 48 a c d n^2 + 48 a^2 c d n^2 + 
 40 b c d n^2 - 80 a b c d n^2 + 28 b^2 c d n^2 + 16 c^2 d n^2 - 
 32 a c^2 d n^2 + 28 b c^2 d n^2 + 4 c^3 d n^2 - 4 d^2 n^2 + 
 24 a d^2 n^2 - 24 a^2 d^2 n^2 - 12 b d^2 n^2 + 24 a b d^2 n^2 - 
 4 b^2 d^2 n^2 - 12 c d^2 n^2 + 24 a c d^2 n^2 - 16 b c d^2 n^2 - 
 4 c^2 d^2 n^2 + 4 a n^3 - 12 a^2 n^3 + 8 a^3 n^3 - 6 b n^3 + 
 36 a b n^3 - 36 a^2 b n^3 - 14 b^2 n^3 + 28 a b^2 n^3 - 4 b^3 n^3 - 
 6 c n^3 + 36 a c n^3 - 36 a^2 c n^3 - 48 b c n^3 + 96 a b c n^3 - 
 40 b^2 c n^3 - 14 c^2 n^3 + 28 a c^2 n^3 - 40 b c^2 n^3 - 
 4 c^3 n^3 + 8 d n^3 - 48 a d n^3 + 48 a^2 d n^3 + 40 b d n^3 - 
 80 a b d n^3 + 24 b^2 d n^3 + 40 c d n^3 - 80 a c d n^3 + 
 64 b c d n^3 + 24 c^2 d n^3 - 16 d^2 n^3 + 32 a d^2 n^3 - 
 16 b d^2 n^3 - 16 c d^2 n^3 - 4 n^4 + 24 a n^4 - 24 a^2 n^4 - 
 28 b n^4 + 56 a b n^4 - 20 b^2 n^4 - 28 c n^4 + 56 a c n^4 - 
 64 b c n^4 - 20 c^2 n^4 + 32 d n^4 - 64 a d n^4 + 48 b d n^4 + 
 48 c d n^4 - 16 d^2 n^4 - 16 n^5 + 32 a n^5 - 32 b n^5 - 32 c n^5 + 
 32 d n^5 - 16 n^6,$ 

 \ 

\noindent $p_{2}(n) = -4 a n + 4 a^3 n + 4 b n - 12 a^2 b n + 12 a b^2 n - 4 b^3 n + 
 4 c n - 12 a^2 c n + 24 a b c n - 12 b^2 c n + 12 a c^2 n - 
 12 b c^2 n - 4 c^3 n - 4 d n + 12 a^2 d n - 24 a b d n + 
 12 b^2 d n - 24 a c d n + 24 b c d n + 12 c^2 d n + 12 a d^2 n - 
 12 b d^2 n - 12 c d^2 n + 4 d^3 n + 12 n^2 - 16 a n^2 - 36 a^2 n^2 + 
 16 a^3 n^2 + 16 b n^2 + 72 a b n^2 - 48 a^2 b n^2 - 36 b^2 n^2 + 
 48 a b^2 n^2 - 16 b^3 n^2 + 16 c n^2 + 72 a c n^2 - 48 a^2 c n^2 - 
 72 b c n^2 + 96 a b c n^2 - 48 b^2 c n^2 - 36 c^2 n^2 + 
 48 a c^2 n^2 - 48 b c^2 n^2 - 16 c^3 n^2 - 16 d n^2 - 72 a d n^2 + 
 48 a^2 d n^2 + 72 b d n^2 - 96 a b d n^2 + 48 b^2 d n^2 + 
 72 c d n^2 - 96 a c d n^2 + 96 b c d n^2 + 48 c^2 d n^2 - 
 36 d^2 n^2 + 48 a d^2 n^2 - 48 b d^2 n^2 - 48 c d^2 n^2 + 
 16 d^3 n^2 + 48 n^3 + 92 a n^3 - 144 a^2 n^3 + 16 a^3 n^3 - 
 92 b n^3 + 288 a b n^3 - 48 a^2 b n^3 - 144 b^2 n^3 + 48 a b^2 n^3 - 
 16 b^3 n^3 - 92 c n^3 + 288 a c n^3 - 48 a^2 c n^3 - 288 b c n^3 + 
 96 a b c n^3 - 48 b^2 c n^3 - 144 c^2 n^3 + 48 a c^2 n^3 - 
 48 b c^2 n^3 - 16 c^3 n^3 + 92 d n^3 - 288 a d n^3 + 48 a^2 d n^3 + 
 288 b d n^3 - 96 a b d n^3 + 48 b^2 d n^3 + 288 c d n^3 - 
 96 a c d n^3 + 96 b c d n^3 + 48 c^2 d n^3 - 144 d^2 n^3 + 
 48 a d^2 n^3 - 48 b d^2 n^3 - 48 c d^2 n^3 + 16 d^3 n^3 - 60 n^4 + 
 432 a n^4 - 144 a^2 n^4 - 432 b n^4 + 288 a b n^4 - 144 b^2 n^4 - 
 432 c n^4 + 288 a c n^4 - 288 b c n^4 - 144 c^2 n^4 + 432 d n^4 - 
 288 a d n^4 + 288 b d n^4 + 288 c d n^4 - 144 d^2 n^4 - 432 n^5 + 
 432 a n^5 - 432 b n^5 - 432 c n^5 + 432 d n^5 - 432 n^6.$

\end{document}